\documentclass[a4paper,10pt,reqno]{amsart}
\usepackage{amssymb,amsmath,bbold}
\usepackage[colorlinks=true,citecolor=blue,linkcolor=blue]{hyperref}
\usepackage{enumitem}
\usepackage{graphicx}
\usepackage{framed}
\usepackage[margin=1in]{geometry}
\usepackage{tikz}
\usetikzlibrary{arrows,automata,backgrounds,calendar,mindmap,trees}
\usepackage{xcolor}

\def\ovv{\overline}
\newcommand{\M}{\mathbf{M}}

\newcommand{\cc}{\boldsymbol{c}}
\newcommand{\ww}{\boldsymbol{c}_{\infty}}
\newcommand{\CC}{\cc}
\renewcommand{\aa}{\boldsymbol{\alpha}}
\newcommand{\bb}{\boldsymbol{\beta}}
\newcommand{\mm}{\boldsymbol{\mu}}

\begin{document}
\title[Convergence to equilibrium for reaction-diffusion systems]
{Explicit exponential convergence to equilibrium\\ for nonlinear reaction-diffusion systems\\ with detailed balance condition}

\author[K. Fellner, B.Q. Tang] {Klemens Fellner, Bao Quoc Tang}

\address{Bao Quoc Tang \hfill\break
Institute of Mathematics and Scientific Computing, University of Graz, Heinrichstrasse 36, 8010 Graz, Austria}
\email{quoc.tang@uni-graz.at} 

\address{Klemens Fellner \hfill\break
Institute of Mathematics and Scientific Computing, University of Graz, Heinrichstrasse 36, 8010 Graz, Austria}
\email{klemens.fellner@uni-graz.at}

\subjclass[2010]{35B35, 35B40, 35K57, 35Q92}
\keywords{Reaction-Diffusion Systems; Exponential Convergence to Equilibrium; Entropy Method; Chemical Reaction Networks; Detailed Balance Condition}

\begin{abstract}
The convergence to equilibrium of mass action reaction-diffusion systems arising from networks of chemical reactions  is studied. The considered reaction networks are assumed to satisfy the detailed balance condition and have no boundary equilibria. We {{propose a general approach based on the so-called entropy method}}, which is able to quantify with explicitly computable rates the decay of an entropy functional in terms of an entropy entropy-dissipation inequality based on the totality of the conservation laws of the system. 

As a consequence follows convergence to the unique detailed balance equilibrium with explicitly computable convergence rates. The general approach is further detailed for two important example systems: a single reversible reaction involving an arbitrary number of chemical substances and a chain of two reversible reactions arising from enzyme reactions.
\end{abstract}

\maketitle
\numberwithin{equation}{section}
\newtheorem{theorem}{Theorem}[section]
\newtheorem{lemma}[theorem]{Lemma}
\newtheorem{proposition}[theorem]{Proposition}
\newtheorem{definition}{Definition}[section]
\newtheorem{remark}{Remark}[section]
\newtheorem{conjecture}{Conjecture}[section]
\newtheorem*{claim}{Claim}

\tableofcontents

\section{Introduction and main results}\label{sec:1}
			
In this paper, we study exponential convergence to equilibrium with explicitly bounded rates and constants for reaction-diffusion systems arising from chemical reaction networks.
\smallskip
			
The considered reaction-diffusion systems describe networks of chemical reactions according to the {{\it law of mass action}} and under the assumption of the {{\it detailed balance condition}}. {{More precisely}}, we consider $I$ chemical substances $\mathcal{C}_1, \ldots, \mathcal{C}_I$ reacting in $R$ reversible reactions of the form
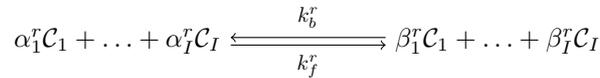
\begin{center}
\begin{tikzpicture}
												 \node (a) {$\alpha_1^r \mathcal{C}_1+ \ldots + \alpha_I^r \mathcal{C}_I$}; \node (b) at (5,0) {$\beta_1^r\mathcal{C}_1+ \ldots + \beta_I^r\mathcal{C}_I$};
												 \draw[arrows=<-]  ([yshift=0.5mm]a.east) -- node [above] {\scalebox{.8}[.8]{$k^r_b$}} ([yshift=0.5mm]b.west) ;
												 \draw[arrows=<-] ([yshift=-0.5mm]b.west) -- node [below] {\scalebox{.8}[.8]{$k^r_f$}} ([yshift=-0.5mm]a.east);
 \end{tikzpicture}
 \end{center}
for $r=1,2,\ldots, R$ with the nonnegative stoichiometric coefficients
$\boldsymbol{\alpha}^r = (\alpha_1^r, \ldots, \alpha_I^r)\in (\{0\}\cup[1,\infty))^{I}$ and $\boldsymbol{\beta}^r = (\beta_1^r, \ldots, \beta_I^r)\in (\{0\}\cup[1,\infty))^{I}$ and  the positive forward and backward reaction rate constants $k^r_f>0$ and $k^r_b>0$. The corresponding reaction-diffusion system for the 
concentration vector $\cc = (c_1, \ldots, c_I):\Omega\times[0,+\infty) \rightarrow [0,+\infty)^{I}$ subject to homogeneous Neumann boundary conditions reads as
\begin{equation}\label{SS}
	\begin{aligned}
&	\frac{\partial}{\partial t}\cc = \mathbb{D}\Delta\cc - \mathbf{R}(\cc), && \text{ in } \Omega\times\mathbb{R}_+,\\
& \nabla\cc\cdot\nu = 0, &&\text{ on } \partial\Omega\times\mathbb{R}_+, \\
& \cc(x,0) = \cc_0(x), &&\text{ on } \Omega,
		\end{aligned}
\end{equation}
where $\Omega\subset \mathbb{R}^{n}$ is a bounded domain with smooth boundary $\partial\Omega$ (e.g. $\partial\Omega\in C^{2+\epsilon}, \epsilon>0$), outward normal unit vector $\nu$ and normalised volume, i.e. 
$$
|\Omega| = 1
$$ 
(this is w.l.o.g. by rescaling the position variable $x\in\Omega$ as $x\to x |\Omega|^{1/n}$). Moreover,  $\mathbb{D} = \mathrm{diag}(d_1, \ldots, d_I)$ is a uniformly positive definite diffusion matrix, i.e. $0<d_{min}\leq d_i \leq d_{max} < +\infty$ for all $i=1,\ldots, I$,  and the reaction vector $\mathbf{R}(\cc)$ 
represents the chemical reactions according to the mass action law, i.e. 
\begin{equation}\label{R}
	\mathbf{R}(\cc) = \sum_{r=1}^{R}(\boldsymbol{\alpha} ^r - \boldsymbol{\beta} ^r)\left(k^r_f\,\cc^{\boldsymbol{\alpha} ^r} - k^r_b\,\cc^{\boldsymbol{\beta} ^r}\right) \quad \text{ with } \quad \cc^{\boldsymbol{\alpha} ^r} = \prod\limits_{i=1}^{I}c_i^{\alpha_i^r} \quad \text{for}\ \ r=1,2,\ldots, R.
\end{equation}
Here, by following e.g. \cite{Vol94}, we observe that the reaction vector $\mathbf{R}(\cc)$ can be written as product of the stoichiometric matrix
\begin{equation*}\label{g4}	
W = \left((\boldsymbol{\beta}^r - \boldsymbol{\alpha}^r)_{r=1,\ldots,R}\right)^{\top }\in \mathbb{R}^{R\times I},
\end{equation*}
which is also called Wegscheider matrix, and the  
reaction rates vector $\mathbf{K}(\cc)$ as modelled according to the mass action law, i.e. 
\begin{equation*}
\mathbf{R}(\cc) = -W^{\top}\mathbf{K}(\cc),
\qquad\text{where}\quad
\mathbf{K}(\cc) = \left(K^r(\cc)\right)_{r=1,\ldots,R}  := \left(k^r_f\,\cc^{\boldsymbol{\alpha} ^r} - k^r_b\,\cc^{\boldsymbol{\beta} ^r}\right)_{r=1,\ldots,R} .
\end{equation*}

The range $\mathrm{rg}(W^{\top})$ is called the {\it stoichiometric subspace} and the above implies that $\mathbf{R}(\cc)\in \mathrm{rg}(W^{\top})$. 
As a consequence, a key structural property of the reaction 
vector $\mathbf{R}(\cc)$ is  
the codim of $W$, which we denote by $m = \mathrm{dim}\;\mathrm{ker}(W)$. 			

If $m>0$, then there exists a (non-unique) matrix $\mathbb{Q}\in \mathbb{R}^{m\times I}$ of zero left-eigenvectors such that 
\begin{equation}\label{Q}
	\mathbb{Q}\,\mathbf{R}(\cc) = 0 \quad \text{ for all } \quad \cc\in \mathbb R^I.
\end{equation}
As a consequence, we have (formally) the following mass conservation laws for \eqref{SS}--\eqref{R} 
\begin{equation}\label{g9}
	\int_{\Omega}\mathbb{Q}\,\cc(t)dx = \int_{\Omega}\mathbb{Q}\,\cc_0\,dx \qquad \text{ or equivalently } \qquad \mathbb{Q}\,\overline{\cc}(t) = \mathbf{M}:= \mathbb{Q}\,\overline{\cc_0} \quad \text{for all}\quad t>0, 
\end{equation}
where $\overline{\cc} = (\overline{c_1}, \ldots, \overline{c_I})$ with $\overline{c_i}(t) = \int_{\Omega}c_i(x,t)dx$ is the spatially averaged concentration vector (recall $|\Omega|=1$) and $\mathbf{M}$ denotes the vector of initial masses, {which can be assumed non-negative, i.e. $\mathbf{M}\in \mathbb{R}^{m}_{\ge0}$ (after changing the sign of the rows of $\mathbb{Q}$, 
for which the corresponding component of $\mathbf{M}$ should be negative).} If $m=0$, then the system \eqref{SS}--\eqref{R} has no conservation law.

Moreover, it is well known that the reaction vector $\mathbf{R}(\cc)$ according to the mass action law, satisfies 
the quasi-positivity condition: If $\mathbf{R}(\cc) = (R_1(\cc), \ldots, R_I(\cc))^{\top}$, then 
\begin{equation*}
\forall  i=1,2,\ldots I: \quad R_i(c_1, \ldots, c_{i-1}, 0, c_{i+1}, \ldots, c_I) \geq 0 \quad \text{ for all } \quad c_1, \ldots, c_{i-1}, c_{i+1}, \ldots, c_I \geq 0. 
\end{equation*}
As a consequence, solutions to \eqref{SS}-\eqref{R}
subject to non-negative initial data $\cc_0\ge0$ 
remain non-negative $\cc(t)\ge0$
for all times $t>0$, see e.g.  \cite{Michel}. 
\medskip

The first main assumption concerning that reactions networks considered in this work is the {\it detailed balance condition}, see e.g. \cite{HJ72,Vol72,Fei79,Vol94}:
A state $\cc^* \in [0,+\infty)^I$ is called a homogeneous equilibrium or shortly an equilibrium of \eqref{SS}--\eqref{R} if and only if
$$
\mathbf{R}(\cc^*) = 0\qquad\text{and}\qquad \mathbb{Q}\,\cc^* = \M.
$$

\begin{itemize}
\item[(\textbf{A1})] System \eqref{SS}--\eqref{R} is assumed to satisfy the \textit{detailed balance condition}, that is there exists an equilibrium $\cc_{\infty}\in (0,+\infty)^I$ such that any forward reaction is balanced with its corresponding backward reaction at this equilibrium, i.e. 
\begin{equation*}
k^r_f\,\cc_{\infty}^{\boldsymbol{\alpha}^r} = k^r_b\,\cc_{\infty}^{\boldsymbol{\beta}^r},\qquad 
 \text{ for all }\ r = 1,2, \ldots, R. 
\end{equation*}
This equilibrium $\cc_{\infty}$ is called a {\it detailed balance equilibrium}.
\end{itemize}
The detailed balance condition allows to rescale the system
\eqref{SS}--\eqref{R} such that we may assume  w.l.o.g.
\begin{equation}\label{kr}
		k^r_f = k^r_b = k^r>0, \qquad \text{ for all } \ r=1,2,\ldots, R.
\end{equation}
This rescaling simplifies the formulation of a second crucial consequence of the detailed balance condition that the logarithmic entropy (or free energy) functional, 
which is the key quantity of our study, 
\begin{equation}\label{entropy}
\mathcal{E}(\cc) = \sum_{i=1}^{I}\int_{\Omega}(c_i\log{c_i} - c_i + 1)\,dx,
\end{equation}
decays monotone in time 
according to the following 
entropy-dissipation functional 
\begin{equation}\label{entropy_dissipation}
\mathcal{D}(\cc) = -\frac{d}{dt}\mathcal{E}(\cc) = \sum_{i=1}^{I}\int_{\Omega}d_i\frac{|\nabla c_i|^2}{c_i}dx + \sum_{r=1}^{R}k^r\int_{\Omega}(\cc^{\boldsymbol{\alpha} ^r} - \cc^{\boldsymbol{\beta} ^r})(\log \cc^{\boldsymbol{\alpha} ^r} - \log \cc^{\boldsymbol{\beta} ^r})\,dx \geq 0.
\end{equation}
\medskip
			
It is well known for detailed balanced reaction networks ({see e.g. \cite{HJ72} for the ODE systems and \cite[Lemma 3.4]{GGH96} for reaction-diffusion systems with homogeneous Neumann boundary conditions}) that for a given positive initial mass vector $\mathbf{M}\in \mathbb{R}^m_{+}$, there exists a unique positive detailed balance equilibrium  $\cc_{\infty}=(c_{1,\infty},\ldots,c_{I,\infty})$ of \eqref{SS}, which is the unique vector of positive constants $\cc_{\infty}>0$ balancing all the reactions
and satisfying the mass conservation laws, i.e.
\begin{equation}\label{positiveequilibrium}
\cc_{\infty}>0\ : \quad \cc_{\infty}^{\boldsymbol{\alpha}^r} = \cc_{\infty}^{\boldsymbol{\beta}^r} \quad \text{for all} \quad r=1,2,\ldots, R
\qquad\text{and}\qquad
\mathbb{Q}\,\cc_{\infty} = \mathbf{M}.
\end{equation}
Note that the existence of a positive detailed balance equilibrium $\cc_{\infty}>0$ can typically also be expected for non-negative initial mass vectors $\mathbf{M}\in \mathbb{R}^{m}_{\ge0}$ with exceptions when $\mathbf{M}=0$.
However, the (chemically meaningless) reversible reaction  $\mathcal{C}_1 \leftrightarrow
2 \mathcal{C}_1 + \mathcal{C}_2$ constitutes an example 
with positive detailed balance equilibrium $\cc_{\infty}=\mathbf{1}$
in the case when the initial mass vector $\mathbf{M}=0$  since $\mathbb{Q}=(1,-1)$ for this system.
\medskip

It is important to remark that besides the unique positive detailed balance equilibrium $\cc_{\infty}>0$, there may also exist additional, so-called boundary equilibria, 
for which $c^*_{i,\infty}=0$ for at least one index $i=1,\ldots,I$. 
See Remark \ref{boundequi} for an example of a system having a boundary equilibrium.

In this paper, we only consider systems with positive detailed balance equilibrium \eqref{positiveequilibrium} and without boundary equilibria. We therefore impose the following second equilibrium assumption:
\begin{itemize}
\item[(\textbf{A2})] The system \eqref{SS}--\eqref{R} features \textit{no boundary equilibrium}, that is \eqref{SS}--\eqref{R} does not possess an equilibrium $\cc^*\in\partial[0,+\infty)^{I}$. 
\end{itemize}

Assumption $(\mathbf{A2})$ is a natural structural condition in order to prove an entropy entropy-dissipation estimate like presented in the following and stated in \eqref{MainEstimate}. In fact, for general systems featuring boundary equilibria, the behaviour near a boundary equilibrium is unclear and boundary equilibria are able to prevent global exponential decay to an asymptotically stable positive detailed balance equilibrium, see e.g. \cite{DFT16}. 
It is also remarked that there exists a large class of systems possessing no boundary equilibria. See e.g. \cite{CF06} for necessary conditions to determine such systems.

\medskip					 

The large time behaviour of solutions to nonlinear reaction-diffusion systems is a highly active research area, which poses many open problems. Classical methods include e.g. linearisation techniques, spectral analysis, invariant regions and Lyapunov stability arguments. 

More recently, the so-called entropy method proved to be a very useful and powerful improvement of classical Lyapunov methods, as it allows, for instance, to show explicit exponential convergence to equilibrium for reaction-diffusion systems. The basic idea of the entropy method consists in studying the large-time asymptotics of a dissipative PDE model by looking for a 
nonnegative convex entropy functional 
$\mathcal E(f)$ and its nonnegative entropy-dissipation functional
$$
		\mathcal D(f)=-\frac{d}{dt} \mathcal E(f(t)) \geq 0
$$ 
along the flow of a PDE model, which is
well-behaved in the following sense: firstly, all states satisfying $\mathcal D(f)=0$ as well as all the involved conservation laws identify a unique entropy-minimising equilibrium $f_{\infty}$, i.e.  
$$
	\mathcal D(f) = 0\quad \text{and }\quad \text{$f$ satisfies all conservation laws}\quad \iff \quad f=f_{\infty},
$$ 
			and secondly, there exists an \emph{entropy entropy-dissipation (EED for short)
			estimate} of the form 
			$$
						\mathcal D(f) \ge \Phi(\mathcal E(f)-\mathcal E(f_{\infty})), \qquad \text{ with } \quad \Phi(x)\ge0, \qquad \Phi(x) = 0 \iff x=0,
			$$ 
for some nonnegative function $\Phi$.  
We remark, that such an inequality can only hold when all the conserved quantities are taken into account.
Moreover, if $\Phi'(0) \neq 0$, a Gronwall argument usually implies exponential convergence toward $f_{\infty}$ in relative entropy $\mathcal E(f)-\mathcal E(f_{\infty})$ with a rate, which can be explicitly estimated. Finally, by applying Csisz\'{a}r-Kullback-Pinsker type inequalities to the relative entropy $\mathcal E(f)-\mathcal E(f_{\infty})$ (recall that $\mathcal E(f)$ is convex), one obtains exponential convergence to equilibrium, for instance, w.r.t. the $L^1$-norm.
			\medskip
			
The entropy method is a fully nonlinear alternative to arguments based on linearisation around the equilibrium and has the advantage of being quite robust with respect to variations and generalisations of the model system. 
This is due to the fact that the entropy method  
relies mainly on functional inequalities which have no direct link to the original PDE model.
	Generalised models typically feature related entropy and entropy-dissipation functionals. Thus, previously established EED estimates may very usefully be re-applied.
			
The entropy method has previously been used for scalar equations: nonlinear diffusion equations (such as fast diffusions \cite{CV,dolbeault}, Landau equation \cite{DVlandauII}),
integral equations (such as the spatially homogeneous 
			Boltzmann equation \cite{toscani_villani1, toscani_villani2, villani_cerc}), 
			kinetic equations (see e.g. \cite{DVinhom1,DVinhom2, fell}), or coagulation-fragmentation models (see e.g. \cite{CDF08,CDF08a}).
			For certain systems of drift-diffusion-reaction equations in semiconductor physics, an entropy entropy-dissipation estimate has been shown in two dimensions indirectly via a compact\-ness-based contradiction argument in \cite{GGH96,GH97}.			
			
The first results for EED estimates for reaction-diffusion systems with explicit rates and constants were established in \cite{DeFe06, DeFe08, DFM08, GZ10, DFIFIP} in particular cases of reversible equations with (at most) 
quadratic nonlinearities. \textcolor{black}{Finally, we refer to e.g. \cite{Gio99,BH,CJ04,CDF08a,BoPie10,CD} for various references where entropy functionals have been importantly used 
in the analysis of general reaction-diffusion systems, for instance in the study of fast-reaction limits. }
\medskip			
			
In this paper, we aim to generalise the entropy method to reaction-diffusion systems with arbitrary mass action law nonlinearities
and, as a consequence, show exponential convergence to equilibrium for \eqref{SS} with explicit bounds on the rates and constants. {{The obtained results extend recent works on the convergence to equilibrium for nonlinear chemical reaction-diffusion systems, see e.g. \cite{MHM, FL, BFL}}}.

Before stating our results, we remark that for general nonlinear reaction-diffusion systems of the form \eqref{SS}--\eqref{R}, the existence of global classical or weak solutions is often an open problem, especially in higher space dimensions and for super-quadratic nonlinearities (see e.g. the survey \cite{Michel} and Remark \ref{existence} for a more detailed discussion).  
This is due to the lack of sufficiently strong a-priori estimates  in order to control nonlinear terms (comparison principles do not hold except for special systems). 

Recently, Fischer \cite{Fi14} proved the global existence of so-called ``renormalised solution" for general reaction-diffusion systems, which dissipate the entropy \eqref{entropy}, and thus provided the existence of global renormalised solutions of system \eqref{SS}--\eqref{R} for arbitrary stoichiometric coefficients. 

\begin{proposition}[Global renormalised solutions to mass action reaction-diffusion systems \cite{Fi14}]\label{renormalised_sol} \hfil\\Let $\Omega$ be a bounded domain in $\mathbb R^n$ with smooth boundary $\partial\Omega$. Assume that the diffusion matrix $\mathbb D$ is positive definite, i.e. $d_i>0$ for all $i=1,\ldots, I$. Let $\cc_0 \in L^1(\Omega)^I$ be nonnegative initial data with $\mathcal E(\cc_0) < +\infty$.

	Then, there exists  a global in time {non-negative} renormalised solution $\cc$ to \eqref{SS}--\eqref{R}, i.e. $0\le c_i \in L^{\infty}_{loc}([0,+\infty); L^1(\Omega))$ and  $\sqrt{c_i} \in L^2_{loc}([0,+\infty); H^1(\Omega))$, and for every smooth function $\xi: (\mathbb R_{+})^I \rightarrow \mathbb R$ with compactly supported derivative $\nabla\xi$ and for every testfunction $\psi\in C^{\infty}(\overline{\Omega}\times [0,+\infty))$ the equation
	\begin{equation}\label{renormalised}
		\begin{aligned}
			\int_{\Omega}\xi(\cc(\cdot, T))\psi(\cdot, T)&\,dx - \int_{\Omega}\xi(\cc_0)\psi(\cdot, 0)\,dx - \int_{0}^{T}\!\!\int_{\Omega}\xi(\cc)\frac{d}{dt}\psi\, dx dt\\
			= &-\sum_{i,j=1}^{I}\int_{0}^{T}\!\!\int_{\Omega}\psi \,\partial_{i}\partial_{j}\xi(\cc)(d_i\nabla c_i)\!\cdot\! \nabla c_j\, dx dt\\
			& - \sum_{i=1}^{I}\int_{0}^{T}\!\!\int_{\Omega}\partial_{i}\xi(\cc)(d_i\nabla c_i)\!\cdot\! \nabla \psi\, dx dt
			 - \sum_{i=1}^{I}\int_{0}^{T}\!\!\int_{\Omega}\partial_i\xi(\cc)R_i(\cc)\psi\, dx dt
		\end{aligned}
	\end{equation}
	holds for almost every $T>0$, with $\mathbf R(\cc) = (R_1(\cc), \ldots, R_I(\cc))$.
	
\end{proposition}
Note that the regularity of renormalised solutions is in general insufficient 
to guarantee the $L^1$-integrability of the reaction-terms on the right hand side of \eqref{SS} and in the entropy dissipation \eqref{entropy_dissipation}. However, sufficiently integrable renormalised solutions are weak solutions, see e.g. 
\cite{DFPV,Fi14}.
\medskip

%

{\color{black} In this paper, we apply the entropy method to show that any renormalised solution to a large class of reaction-diffusion systems of the form \eqref{SS} converges exponentially to equilibrium with computable rates.}

\medskip
The key lemma of the entropy method 
is the following EED functional inequality
\begin{equation}\label{MainEstimate}
						\boxed{\mathcal{D}(\cc) \geq \lambda(\mathcal{E}(\cc) - \mathcal{E}(\cc_{\infty}))}
\end{equation}
for all $\cc\in L^1(\Omega; [0,+\infty)^I)$ obeying the mass conservation $\mathbb{Q}\,\overline{\cc} = \mathbf{M}$, and where $\lambda = \lambda(\Omega, \mathbb D, \M, \boldsymbol{\alpha}^r, \boldsymbol{\beta}^r)$ is a constant depending on the domain $\Omega$, the diffusion coefficients $\mathbb D$, the initial mass vector $\M$ and the stoichiometric coefficients.

Let's assume for the moment that the functional inequality \eqref{MainEstimate} is proven. 
\textcolor{black}{The existence of global renormalised solutions to \eqref{SS} follows from \cite{Fi14} provided that the {\it detailed balanced condition } holds, which guarantees the free energy functional \eqref{entropy}. The existence of global renormalised solution was in fact proven for more general, entropy dissipating systems (see \cite{Fi14} for more details). Moreover, in \cite{Fis16}, it was recently shown that any renormalised solution according to Proposition \ref{renormalised_sol} satisfies the following weak entropy entropy-dissipation inequality as a generalised version of \eqref{entropy_dissipation}:}
\begin{equation}\label{weak_entropy}
	\mathcal{E}(\cc(t)) + \int_{t_0}^t\mathcal{D}(\cc(s))ds \leq \mathcal{E}(\cc(t_0)) \quad \text{ for a.a. } 0\le t_0<t.
\end{equation}
Thus, by applying the functional inequality \eqref{MainEstimate} to these renormalised solutions of \eqref{SS}--\eqref{R}, a Gronwall argument (see e.g. \cite{Wil,FL}) yields exponential convergence in relative entropy with the rate $\lambda$, where $\lambda$ is given in \eqref{MainEstimate} and can be explicitly estimated.  
Moreover, by applying a Csisz$\acute{\mathrm{a}}$r-Kullback-Pinsker type inequality (see Section \ref{sec:2}) one also obtains $L^1$-convergence to equilibrium of solutions to \eqref{SS} with the rate $e^{-\lambda t/2 }$.
\medskip

In \cite{MHM}, by using an inspired convexification argument and under the assumption of the detailed balance condition, the authors proved that \eqref{MainEstimate} holds  for systems \eqref{SS}--\eqref{R} for a $\lambda>0$ provided that the detailed balance equilibrium \eqref{positiveequilibrium} is the only equilibrium and that there are no boundary equilibria. 
Moreover, they gave an explicit bound of $\lambda$ in the case of the quadratic reaction $2\mathcal{C}_1 \leftrightharpoons \mathcal{C}_2$. However, because of the non-convex structure of the problem, obtaining explicit estimates on  $\lambda$ via convexification seems difficult in the case of more than two substances, e.g.
for systems like
\begin{equation*}
\alpha \mathcal{C}_1 + \beta \mathcal{C}_2 \leftrightharpoons \gamma \mathcal{C}_3 \qquad \text{ or } \qquad \mathcal{C}_1 + \mathcal{C}_2 \leftrightharpoons \mathcal{C}_3 + \mathcal{C}_4.
\end{equation*}
			
By drawing from various previous ideas in \cite{DeFe08, DFIFIP, FL, BFL}, this paper aims to propose a constructive way to prove {{{\it quantitatively}}} the EED estimate \eqref{MainEstimate} for general mass action law reaction-diffusion systems. The main advantage of our method is that, by extensively using the structure of the mass conservation laws, the proof relies on   elementary inequalities and has the advantage of providing explicit estimates for the convergence rate $\lambda$. 

Another advantage of the here-proposed method is its robustness in the sense that it also applies to (bio-)chemical reaction networks where substances are supported on different compartments. In Subsection \ref{further}, we show that our method directly generalises to a specific example of a volume-surface reaction-diffusion system. Volume-surface reaction-diffusion systems are recently becoming highly relevant models with many applications in cell-biology 
\cite{NGCRSS07,MS11,FNR13,FRT16,BFL} or also crystal growth \cite{KD01}.  
We remark that the method of convexification as presented in \cite{MHM} seems not to apply to such volume-surface reaction-diffusion systems. 
\textcolor{black}{In fact, the convexification method estimates first the non-convex reaction terms inside the entropy-dissipation integral below in terms of a convexified version. Secondly, this integral term
is then further estimated below by Jensen's inequality to obtain a convex function depending on spatial averages the concentrations. For volume-surface reactions, however, such an approach would partially result in 
surface averages of traces of volume-concentrations, which seems not helpful in controlling volume-averages.}
\medskip

The first two main results of this paper will detail the proposed method for two important model systems: a single reversible reaction with arbitrary number of substances of the form
\begin{equation}\label{ReSingle}
\alpha_1\mathcal{A}_1 + \ldots + \alpha_I\mathcal{A}_I \leftrightharpoons \beta_1 \mathcal{B}_1 + \ldots + \beta_J\mathcal{B}_J
\end{equation}
and a chain of two reversible reactions, which generalises 
the Michaels-Menton model for catalytic enzyme kinetics (see e.g. \cite{Mur})
\begin{equation}\label{ReChain}
\mathcal{C}_1 + \mathcal{C}_2 \leftrightharpoons \mathcal{C}_3 \leftrightharpoons \mathcal{C}_4 + \mathcal{C}_5.
\end{equation}

{
Note that for the single reversible reaction \eqref{ReSingle}, it is more convenient and consistent with the literature to change the notation compared to \eqref{SS} by splitting the concentration vector $\boldsymbol{c}$
into a left-hand-side and a right-hand-side-concentration vector, i.e. 
$$
\boldsymbol{c} = (c_1, \ldots, c_I) \quad\to\quad (\boldsymbol{a},\boldsymbol{b}) = (a_1, \ldots, a_I, b_1, \ldots, b_J),
$$
where $I$ denotes now the number of left-hand-side concentrations and $J$  the number of right-hand-side concentrations. This notation allows a clearer presentation of the corresponding system and the proofs.}

\medskip
At first, after assuming (w.l.o.g.) that the forward and backward reaction rate constants are normalised to one, the mass action reaction-diffusion system modelling \eqref{ReSingle} reads as 
\begin{equation}\label{SysSingle}
		\begin{cases}
		\partial_ta_i - d_{a,i}\Delta a_{i} = -\alpha_i(\boldsymbol{a}^{\boldsymbol{\alpha}} - \boldsymbol{b}^{\boldsymbol{\beta}}), &\quad \text{ in } \Omega\times\mathbb{R}_+,\quad i=1,2,\ldots, I, \\
	\partial_tb_j- d_{b,j}\Delta b_{j} = \beta_j(\boldsymbol{a}^{\boldsymbol{\alpha}} - \boldsymbol{b}^{\boldsymbol{\beta}}), & \quad \text{ in } \Omega\times\mathbb{R}_+,\quad j=1,2,\ldots, J,\\
	\nabla a_i\cdot \nu = \nabla b_j\cdot \nu = 0, &\quad \text{ on } \partial\Omega\times\mathbb{R}_+,\quad i=1,\ldots, I,\; j=1,\ldots, J,\\
	\boldsymbol{a}(x,0) = \boldsymbol{a}_0(x), \quad \boldsymbol{b}(x,0) = \boldsymbol{b}_0(x), &\quad \text{ in }\Omega,
	\end{cases}
\end{equation}
where $\boldsymbol{a} = (a_1, \ldots, a_I)$ and $\boldsymbol{b} = (b_1, \ldots, b_J)$ 
denote the two vectors for left- and right-hand side concentrations, $d_{a,i}>0$ and $d_{b,j}>0$ are the positive diffusion coefficients,  
and $\boldsymbol{\alpha} = (\alpha_1, \ldots, \alpha_I)\in [1,\infty)^{I}$ and 
$\boldsymbol{\beta} = (\beta_1, \ldots, \beta_J)\in [1,\infty)^{I}$ are the positive vectors of the stoichiometric 
coefficients assossiated to the single reaction \eqref{ReSingle}.
Recall that  $\boldsymbol{a}^{\boldsymbol{\alpha}}=\prod_{i=1}^{I}a_i^{\alpha_i}$ and 
$\boldsymbol{b}^{\boldsymbol{\beta}}=\prod_{j=1}^{J}b_j^{\beta_j}$.
%

The system \eqref{SysSingle} 
possesses the following $IJ$ mass conservation laws
\begin{equation}\label{MassSingle}
	\frac{\overline{a_i}}{\alpha_i} + \frac{\overline{b_j}}{\beta_j} = M_{i,j},\qquad i=1,\ldots, I, \ \ j=1,\ldots, J,
\end{equation}
from which exactly 
{
$m = I+J-1$ conservation laws are linear independent. That means the matrix $\mathbb{Q}$ in this case has the dimension $\mathbb{Q}\in \mathbb{R }^{(I+J-1)\times (I+J)}$. See Lemma \ref{masssingle} below for an explicit form of $\mathbb{Q}$.}
After choosing and fixing $I+J-1$ linear independent components 
from the $IJ$ {conserved masses $(M_{i,j}) \in \mathbb{R}_{+}^{IJ}$, 
we denote by 
$\mathbf{M} \in \mathbb{R}^{I+J-1}$ the vector of initial masses corresponding to the selected $I+J-1$ coordinates of 
$(M_{i,j})\in \mathbb{R}^{IJ}_{+}$.
Thus, by a given initial mass vector $\mathbf{M}$, we signify that these $I+J-1$ coordinates are given and the remaining coordinates of $(M_{i,j})\in \mathbb{R}^{IJ}_{+}$ are subsequently calculated.}
The unique positive detailed balance equilibrium $(\boldsymbol{a}_{\infty}, \boldsymbol{b}_{\infty})\in \mathbb{R}^{I+J}_+$ of \eqref{SysSingle} is defined by
\begin{equation*}
		\begin{cases}
		\frac{a_{i,\infty}}{\alpha_i} + \frac{b_{j,\infty}}{\beta_j} = M_{i,j} \qquad \forall i=1,2,\ldots, I,\; \forall j=1,2,\ldots, J,\\[2mm]
		\boldsymbol{a}_{\infty}^{\boldsymbol{\alpha}} = \boldsymbol{b}_{\infty}^{\boldsymbol{\beta}}.
		\end{cases}
\end{equation*}

The corresponding entropy  and entropy-dissipation functionals for system \eqref{SysSingle} are
\begin{equation}\label{l1}
			\mathcal{E}(\boldsymbol{a}, \boldsymbol{b}) = \sum_{i=1}^{I}\int_{\Omega}(a_i\log{a_i} - a_i + 1)dx + \sum_{j=1}^{J}\int_{\Omega}(b_j\log{b_j} - b_j + 1)dx
\end{equation}
and 
\begin{equation}\label{l2}
			\mathcal{D}(\boldsymbol{a}, \boldsymbol{b}) = \sum_{i=1}^{I}\int_{\Omega} d_{a,i}\frac{|\nabla a_i|^2}{a_i}dx + \sum_{j=1}^{J}\int_{\Omega} d_{b,j}\frac{|\nabla b_j|^2}{b_j}dx + \int_{\Omega}(\boldsymbol{a}^{\boldsymbol{\alpha}} - \boldsymbol{b}^{\boldsymbol{\beta}})\log{\frac{\boldsymbol{a}^{\boldsymbol{\alpha}}}{\boldsymbol{b}^{\boldsymbol{\beta}}}}dx,
\end{equation}
respectively.

\begin{theorem}[Explicit convergence to equilibrium for a single reversible reaction \eqref{ReSingle}]\label{mainsingle}\hfill\\
Let $\Omega\subset \mathbb R^n$ be a bounded domain with smooth boundary. Assume positive diffusion coefficients $d_{a,i}>0$ and $d_{b,j}>0$ for all $i=1,\ldots, I$ and $j=1,\ldots, J$. Assume stoichiometric coefficients $\alpha_{i}\geq 1$ and $\beta_j\geq 1$ for all $i=1,\ldots, I$ and $j=1,\ldots, J$.


Suppose assumptions $(\mathbf{A1})$ and $(\mathbf{A2})$, i.e. that system \eqref{SysSingle} features a unique positive detailed balance equilibrium $(\boldsymbol{a}_{\infty}, \boldsymbol{b}_{\infty})$ of the form \eqref{positiveequilibrium} and no boundary equilibria. 
Note that for a given positive initial mass vector $\M\in \mathbb{R}^{I+J-1}_+$
corresponding to $I+J-1$ linear independent conservation laws \eqref{MassSingle}, the assumptions $(\mathbf{A1})$ and $(\mathbf{A2})$ are in fact a consequence of Lemma \ref{Equilibrium} and thus satisfied.

\medskip
Then, for any nonnegative functions $(\boldsymbol{a}, \boldsymbol{b}) \in L^1\left(\Omega; [0,\infty)^{I+J}\right)$ satisfying the mass conservation laws \eqref{MassSingle}, the following EED estimate
\begin{equation}\label{EEDsingle}
\mathcal{D}(\boldsymbol{a}, \boldsymbol{b}) \geq \lambda_1(\mathcal{E}(\boldsymbol{a}, \boldsymbol{b}) - \mathcal{E}(\boldsymbol{a}_{\infty}, \boldsymbol{b}_{\infty}))
\end{equation}
holds {{for $\mathcal{E}(\boldsymbol{a}, \boldsymbol{b})$ and $\mathcal{D}(\boldsymbol{a}, \boldsymbol{b})$ defined in \eqref{l1} and \eqref{l2}, respectively}}, and where the constant $\lambda_1 > 0$ can be explicitly estimated in terms of the initial mass vector $\M$, the domain $\Omega$, the stoichiometric coefficients $\boldsymbol{\alpha}, \boldsymbol{\beta}$ and the diffusion coefficients $d_{a,i}$ and $d_{b,j}$.						
\medskip

Furthermore, for any nonnegative initial data $(\boldsymbol{a}_0, \boldsymbol{b}_0)\in L^1(\Omega)^{I+J}$ with finite entropy, i.e. $\mathcal {E}(\boldsymbol{a}_0,\boldsymbol{b}_0) <+\infty$, there exist global non-negative renormalised solutions $(\boldsymbol{a}, \boldsymbol{b})$ of \eqref{SysSingle} in the 
sense of Proposition \ref{renormalised_sol}. Moreover, 
\textcolor{black}{any renormalised solution to \eqref{SysSingle}  
satisfies the mass conservation laws \eqref{MassSingle} (see Proposition \ref{mass_renormalised})
and 
the weak entropy entropy-dissipation law \eqref{weak_entropy}, i.e. (see \cite{Fis16})}
\begin{equation*}
	\mathcal{E}(\boldsymbol{a}(t), \boldsymbol{b}(t)) + \int_{t_0}^t\mathcal{D}(\boldsymbol{a}(s), \boldsymbol{b}(s))ds \leq \mathcal{E}(\boldsymbol{a}(t_0), \boldsymbol{b}(t_0)) \quad \text{ for a.a. } 0\le t_0<t.
\end{equation*}

Consequently, any renormalised solution of \eqref{SysSingle} 
converges exponentially to the detailed balance equilibrium $(\boldsymbol{a}_{\infty}, \boldsymbol{b}_{\infty})$ in $L^1$-norm with the rate $\frac{\lambda_1}{2}$ (as stated in \eqref{EEDsingle}), i.e.
%
\begin{equation*}
	\sum_{i=1}^{I}\|a_i(t) - a_{i,\infty}\|_{L^1(\Omega)}^2 + \sum_{j=1}^{J}\|b_j(t) - 	b_{j,\infty}\|_{L^1(\Omega)}^2
	\leq C_{CKP}^{-1}(\mathcal{E}(\boldsymbol{a}_0, \boldsymbol{b}_0) - \mathcal{E}(\boldsymbol{a}_{\infty}, \boldsymbol{b}_{\infty}))e^{-\lambda_1 t}
\end{equation*}
where $C_{CKP}$ is the constant in a Csisz$\acute{a}$r-Kullback-Pinsker inequality in Lemma \ref{CKP}.
\end{theorem}
\begin{remark}\label{existence}
Theorem \ref{mainsingle} is formulated for renormalised solution, which is the only available general concept of global solutions for nonlinear entropy-dissipating reaction-diffusion systems like \eqref{SS}--\eqref{R}. {\color{black} All renormalised solutions are shown in \cite[Proposition 6]{Fis16} to satisfy a weak entropy entropy-dissipation law, yet not in terms of an equality like \eqref{entropy_dissipation}, but in terms of the inequality \eqref{weak_entropy}. Nevertheless, this is sufficient to obtain exponential convergence to equilibrium via a Gronwall argument, see \cite{Wil,FL}.
\textcolor{black}{
Moreover, Proposition \ref{mass_renormalised} shows that 
all renormalised solutions of \eqref{SysSingle} according to Proposition \ref{renormalised_sol} satisfy the conservation laws \eqref{MassSingle}, which are required to 
identify the detailed balance equilibrium $(\boldsymbol{a}_{\infty}, \boldsymbol{b}_{\infty})$. We refer to Remark \ref{massrenorm} below stating that mass conservation remains an open problem for renormalised solutions of some detailed balance systems.}


For classical and weak solutions of \eqref{SS}--\eqref{R}, however, the weak entropy entropy-dissipation law \eqref{weak_entropy} can typically be verified with an equality sign.}
This was done, for instance, in \cite{DFPV} for weak $(L\, \log L)^2$-solutions of a system with quadratic nonlinearities (see also Theorem \ref{mainChain}).
	More recently, global weak solutions satisfying \eqref{weak_entropy} were shown to exist for  systems of the form \eqref{SS}--\eqref{R} in all space dimensions provided a (dimension- and nonlinearity-dependent)  "closeness'' assumption on the diffusion coefficients, see e.g. 
	\cite{FL}. Imposing a {stronger}  "closeness'' assumption on the diffusion coefficients allows to even show the existence of global classical solutions, see e.g. \cite{FLS16}.
\end{remark}	
\begin{remark}
	Theorem \ref{mainsingle} generalises the previous results of \cite{DeFe06, DeFe08,GZ10,MHM, FL}, where only special cases of system \eqref{SysSingle} were treated.
\end{remark}		
			
As second example \eqref{ReChain}, after assuming ({for the sake of simplicity}) that all the forward and backward reaction rate constants are normalised to one, the reaction-diffusion system modelling \eqref{ReChain} reads (by reverting back to the general notation of system \eqref{SS}) as
\begin{equation}\label{SysChain}
	\begin{cases}
		\partial_tc_1 - d_1\Delta c_1 = -c_1c_2 + c_3, &\text{ in } \Omega\times\mathbb{R}_+, \\
		\partial_tc_2 - d_2\Delta c_2 = -c_1c_2 + c_3, &\text{ in } \Omega\times\mathbb{R}_+, \\
		\partial_tc_3 - d_3\Delta c_3 = c_1c_2 + c_4c_5 - 2c_3, \quad&\text{ in } \Omega\times\mathbb{R}_+,\\
		\partial_tc_4 - d_4\Delta c_4 = -c_4c_5 + c_3, &\text{ in } \Omega\times\mathbb{R}_+,\\
		\partial_tc_5 - d_5\Delta c_5 = -c_4c_5 + c_3, &\text{ in } \Omega\times\mathbb{R}_+,\\
		\nabla c_i\cdot \nu =  0, \qquad\qquad i=1,2,\ldots, 5,\quad &\text{ on } \partial\Omega\times\mathbb{R}_+,\\
		c_i(x,0) = c_{i,0}(x), &\text{ in } \Omega,
	\end{cases}
\end{equation}
where $d_i >0$, $i=1,\ldots 5$, are positive diffusion coefficients.
The four mass conservation laws of \eqref{SysChain} are
\begin{equation}\label{MassChain}
		\overline{c_i} + \overline{c_3} + \overline{c_j} = M_{i,j}, \qquad \forall i\in \{1,2\} \quad\text{ and }\quad \forall j\in \{4,5\}
\end{equation}
and among these there are 
{
$m = 3$ linear independent conservation laws, thus $\mathbb{Q}\in \mathbb{R}^{3\times 5}$.} 
In the following, we denote by $\boldsymbol{c} = (c_1, \ldots, c_5)$ the concentration vector and by 
$(M_{i,j}) = (M_{1,4}, M_{1,5}, M_{2,4}, M_{2,5})\in \mathbb{R}^4$  the {vector of conserved masses}. 
Note that the initial mass vector $\M$ is determined by any three coordinates of $(M_{i,j})\in \mathbb{R}^4_{+}$ corresponding to three linear independent 
conservation laws and by a given vector $(M_{i,j})\in \mathbb{R}^4_{+}$ we mean that these three coordinates are given and the remaining coordinate is subsequently calculated. 
The unique positive detailed balance equilibrium 
$\boldsymbol{c}_{\infty} = (c_{1,\infty}, \ldots, c_{5,\infty})\in \mathbb{R}^{5}$ to \eqref{SysChain} is defined  by
\begin{equation}\label{PosEquiChain}
\begin{cases}
									c_{i,\infty} + c_{3,\infty} + c_{j,\infty} = M_{i,j}, \qquad \forall i\in \{1,2\} \quad \text{ and }\quad \forall j\in \{4,5\},\\
									c_{1,\infty}c_{2,\infty} = c_{3,\infty},\\
									c_{4,\infty}c_{5,\infty} = c_{3,\infty}.
\end{cases}
\end{equation}
			
The corresponding entropy  and entropy-dissipation functionals for system \eqref{SysChain} are
\begin{equation}\label{l3}
\mathcal{E}(\boldsymbol{c}) = \sum_{i=1}^{5}\int_{\Omega}(c_i\log{c_i} - c_i + 1)dx
\end{equation}
and 
\begin{equation}\label{l4}
\mathcal{D}(\boldsymbol{c}) = \sum_{i=1}^{5}\int_{\Omega}d_i\frac{|\nabla c_i|^2}{c_i}dx + \int_{\Omega}\left((c_1c_2 - c_3)\log{\frac{c_1c_2}{c_3}} + (c_4c_5 - c_3)\log{\frac{c_4c_5}{c_3}}\right)dx,
\end{equation}
respectively.

\begin{theorem}[Explicit convergence to equilibrium for the chain reaction \eqref{ReChain}]\label{mainChain}\hfill\\
Let $\Omega \subset \mathbb R^n$ be a bounded domain with smooth boundary. Assume positive diffusion coefficients  $d_i>0$ for all $i=1,\ldots, 5$.

	
Suppose assumptions $(\mathbf{A1})$, i.e. that 
system \eqref{SysChain} has a unique positive detailed balance equilibrium \eqref{PosEquiChain} and observe that system \eqref{SysChain} has no boundary equilibria.
Note that for a given positive initial mass vector $\M\in \mathbb{R}^{3}_+$
corresponding to three linear independent conservation laws \eqref{MassChain}, the assumption $(\mathbf{A1})$ is in fact a consequence of Lemma \ref{Chain_Equi} and thus satisfied.

Then, for any nonnegative function $\boldsymbol{c} = (c_1, \ldots, c_5) \in L^1\left(\Omega; [0,+\infty)^5\right)$ satisfying the mass conservation laws \eqref{MassChain}, 
the EED estimate
\begin{equation}\label{l5}
\mathcal{D}(\boldsymbol{c}) \geq \lambda_2(\mathcal{E}(\boldsymbol{c}) - \mathcal{E}(\boldsymbol{c}_{\infty}))
\end{equation}
holds {{for $\mathcal{E}(\boldsymbol{c})$ and $\mathcal{D}(\boldsymbol{c})$ defined in \eqref{l3} and \eqref{l4} respectively}}, and where $\lambda_2>0$ is a positive constant which can be {\normalfont{explicitly}} estimated in terms of the initial mass vector $\M$, the domain $\Omega$ and the diffusion coefficients $d_{i}, i=1,2,\ldots, 5$.
						
						\medskip
Moreover, for any nonnegative initial data $\boldsymbol c_0 \in L^p(\Omega)^5$ for some $p>2$ (sufficiently close to 2), there exists a global $L^p$-weak solution $\cc$ to \eqref{SysChain}. These weak solutions satisfy the conservation laws \eqref{MassChain} and 
the weak entropy entropy-dissipation law,
\begin{equation*}
\mathcal{E}(\cc(t)) + \int_{t_0}^t\mathcal{D}(\cc(s))ds = \mathcal{E}(\cc(t_0))\quad \text{ for a.a. } 0\le t_0<t.
\end{equation*}
Consequently, all these weak solutions converge exponentially to the equilibrium, i.e.
\begin{equation*}
									\sum_{i=1}^{5}\|c_i(t) - c_{i,\infty}\|_{L^1(\Omega)}^2 \leq C_{CKP}^{-1}(\mathcal{E}(\boldsymbol{c}_0) - \mathcal{E}(\boldsymbol{c}_{\infty}))e^{-\lambda_2 t}, \qquad \forall t>0,
						\end{equation*} 
						where $\lambda_2$ is in \eqref{l5} and $C_{CKP}$ is the constant in the Csisz$\acute{a}$r-Kullback-Pinsker inequality.
\end{theorem}

\medskip

The proofs of the above Theorems \ref{mainsingle} and \ref{mainChain} -- in particular the corresponding EED estimates --
are based on the following If-Theorem \ref{the:main}, which can be understood as a proof of concept of how to derive explicit EED estimates for general mass-action-law detailed balance reaction-diffusion systems. 
More precisely, the following Theorem \ref{the:main} shows that provided the conservation laws of a suitable reaction-diffusion systems are sufficiently explicitly given to prove two natural key inequalities, 
then an EED estimate with explicitly estimable constants and rates follows from a general method of proof.

\begin{theorem}[Entropy entropy-dissipation estimates for general detailed balance RD networks \eqref{SS}]\label{the:main}\hfill\\
Let $\Omega\subset \mathbb R^n$ be a bounded domain with smooth boundary. Assume positive diffusion coefficients $d_{i}>0$ for all $i=1,\ldots, I$. Assume stoichiometric coefficients $\boldsymbol{\alpha}^r = (\alpha_1^r, \ldots, \alpha_I^r)\in (\{0\}\cup[1,\infty))^{I}$ and $\boldsymbol{\beta}^r = (\beta_1^r, \ldots, \beta_I^r)\in (\{0\}\cup[1,\infty))^{I}$ and normalised reaction rate constants $k^r>0$ as rescaled in \eqref{kr} for all $r=1,2,\ldots, R$.
Suppose assumptions $(\mathbf{A1})$ and $(\mathbf{A2})$, i.e. that system \eqref{SS}--\eqref{R} features a unique positive detailed balance equilibrium $\boldsymbol{c}_{\infty}$ of the form \eqref{positiveequilibrium} and no boundary equilibria. 
The positivity of the detailed balance equilibrium follows, for instance, from supposing 
a positive initial mass vector $\M\in \mathbb{R}^{m}_+$
corresponding to $m$ linear independent conservation laws \eqref{g9}, {see e.g. \cite[Lemma 3.4]{GGH96}.}

Moreover, assume
\begin{itemize}
\item[i)] that for all bounded states $\ovv{\boldsymbol{c}}\in [0,K]^{I}$ (for a $K>0$ sufficiently large {\color{black} depending only on the initial relative entropy $\mathcal{E}(\cc_0|\ww)$}), which satisfy the conservation laws
\begin{equation*}
\mathbb Q\, \ovv{\cc} = \M,
\end{equation*}
there exists a constant $H_4>0$ such that
\begin{equation}\label{ODEEqui}
\sum_{r=1}^{R}\left[\sqrt{\frac{\ovv{\cc}}{\cc_\infty}}^{\,\aa^r} - \sqrt{\frac{\ovv{\cc}}{\cc_\infty}}^{\,\bb^r}\right]^2 \geq H_4\sum_{i=1}^{I}\left(\frac{\sqrt{\ovv{c_i}}}{\sqrt{c_{i,\infty}}}-1\right)^2,
\end{equation}
\item[ii)] and 
that there exists $0 < \varepsilon \ll 1$ such that if $\overline{c_{i_0}} \leq \varepsilon^2$ for some $i_0\in \{1, \ldots, I\}$, then we can find $H_5(\varepsilon)>0$ depending on $\varepsilon$ such that
\begin{equation}\label{FarEqui}
				\sum_{i=1}^{I}\|\nabla \sqrt{c_i}\|_{L^2(\Omega)}^2  + \sum_{r=1}^{R}\left(\ovv{\sqrt{\cc}}^{\,\boldsymbol{\alpha}^r} - \ovv{\sqrt{\cc}}^{\,\boldsymbol{\beta}^r}\right)^2 \geq H_5.
\end{equation}
\end{itemize}
\medskip
Then, for the entropy and entropy-dissipation functional
\eqref{entropy} and \eqref{entropy_dissipation},
the key entropy entropy-dissipation inequality \eqref{MainEstimate}, i.e. 
\begin{equation*}
\mathcal{D}(\cc) \geq \lambda(\mathcal{E}(\cc) - \mathcal{E}(\cc_{\infty}))
\end{equation*}
holds true for all functions $\cc\in L^1(\Omega; [0,+\infty)^I)$ obeying the mass conservation $\mathbb{Q}\,\overline{\cc} = \mathbf{M}$, and where $\lambda = \lambda(\Omega, H_4, H_5, \mathbb D, \M, \boldsymbol{\alpha}^r, \boldsymbol{\beta}^r, k^r)$ is a constant depending {\emph{explicitly}} on the domain $\Omega$, the constants  $H_4$ and $H_5$, the diffusion coefficients $\mathbb D$, the stoichiometric coefficients $\aa^r, \bb^r$, the reaction rate constants $k^{r}$, and the initial mass vector $\M$. 
 \end{theorem}

The first assumption \eqref{ODEEqui} in Theorem \ref{the:main} 
can be either interpreted as a quantitative version of the uniqueness 
of the positive detailed balance equilibrium \eqref{positiveequilibrium} or an kind of EED estimate (w.r.t. $L^2$-distances of square-roots of concentrations) for 
the ODE system associated to \eqref{SS}--\eqref{R}. 
Indeed, both viewpoints reflect the observation that 
the positive detailed balance equilibrium $\cc_\infty$  balances all 
reactions (in fact the left hand side of \eqref{ODEEqui} is precisely zero for all states $\ovv{\cc}$, which balance all reactions) and satisfies all conservation law 
$\mathbb Q\, \ovv{\cc} = \M$. These two conditions uniquely define the positive detailed balance equilibrium $\cc_\infty$, which is the only state for which the right hand side of \eqref{ODEEqui} is zero.  
Moreover, considering bounded states 
$0\le\ovv{\cc}\le K$
constitutes no relevant restriction since Lemma \ref{bounded} below will show that all solutions to \eqref{SS}--\eqref{R} with bounded initial relative entropy satisfy naturally uniform-in-time bounds $0\le\ovv{\cc}\le K$ for some sufficiently large constant $K>0$.

We remark that inequality \eqref{ODEEqui} can be shown to hold with an explicit constant via Taylor expansion for states $\ovv{\cc}$ near the equilibrium $\cc_\infty$, yet this is insufficient to prove an explicit bound for the constant $H_4$ for states  $\ovv{\cc}$ far from equilibrium due to the non-convexity of the problem, see \cite{MHM}.

In our proof of the Theorems \ref{mainsingle} and \ref{mainChain}, we present a global method of 
proving inequality \eqref{ODEEqui} with an explicit bound for $H_4$. The key to this proof is to exploit the structure of the corresponding conservations laws \eqref{MassSingle} and \eqref{MassChain}, respectively. 
In fact, for the proof of Theorem \ref{mainsingle}, it is sufficient 
to observe that the conservations laws \eqref{MassSingle} entail some qualitative sign relations between the states 
$\ovv{\cc}$ around the equilibrium $\boldsymbol{c}_{\infty}$. 
\medskip

The second inequality \eqref{FarEqui} is a quantified version of the natural observation that all states $\ovv{\cc}$, for which less than a sufficiently small amount of mass is present in at least one 
concentration $\ovv{{c}_{i_0}}<\varepsilon$ (recall that the detailed balance equilibrium $\boldsymbol{c}_{\infty}$ contains a fixed positive amount of mass in all concentrations), are necessarily far from equilibrium in the sense that left hand side of  \eqref{FarEqui}, which is a lower bound for the entropy dissipation $\mathcal D(\cc)$ (see Section \ref{subsec:2.2}), is itself bounded below by a positive constant $H_5(\varepsilon)$. More precisely, \eqref{FarEqui} can be interpretated that such states are either inhomogeneous in space 
and thus dissipate entropy in terms of the Fisher information (which is the first term on the left hand side of \eqref{FarEqui})
or that such states are still subject to a significant amount of chemical reactions and thus dissipate entropy in terms of the second term on the left hand side of \eqref{FarEqui}. 
The inequality \eqref{FarEqui} also reflects the fact that the considered system possess no boundary equilibria, and that the entropy dissipation is bounded away from zero when the concentration is close to the boundary $\partial\mathbb R_+^{I}$.
\medskip

Finally, we remark our believe that our method of proving
the inequalities \eqref{ODEEqui} and \eqref{FarEqui} in Theorems \ref{mainsingle} and \ref{mainChain}
should also apply to any other systems of the form
\eqref{SS}--\eqref{R}, once the structure of the conservation laws 
is sufficiently explicitly given. The main reason for having to state Theorem \ref{the:main} as an If-Theorem by assuming 
the inequalities \eqref{ODEEqui} and \eqref{FarEqui}
is the fact that we do not know how to provide an explicit proof 
of these two natural inequalities only based on the pure existence of the matrix $\mathbb Q$, without knowledge of 
its structure which formalises the conservation laws.

\medskip

The rest of this paper is organized as follows: In Section \ref{sec:2}, we first provide preliminary estimates and results before proving Theorem \ref{the:main} for general systems of the form \eqref{SS}--\eqref{R}. The proofs of Theorems \ref{mainsingle} and \ref{mainChain} are presented in Sections \ref{sec:3} and \ref{sec:4}, respectively. Finally, we discuss the further possible applications and some open problems in Section \ref{sec:5}.
						
\section{The general method}\label{sec:2}

In this section, we first briefly state preliminary estimates and inequalities of the mass action reaction-diffusion systems 
\eqref{SS}--\eqref{entropy_dissipation} before we present the 
details of our proposed method.

The following notations and elementary inequalities are used in our proof:
\begin{description}[font=\normalfont]
 \item[{\it $L^2(\Omega)$-norm}] \hfill\\
For the rest of this paper, we will denote by $\|\cdot\|$ the usual norm of $L^2(\Omega)$:
$
\|f\|^2 = \int_{\Omega}|f(x)|^2\,dx.\smallskip
$
\item[{\it Spatial averages and square-root abbreviation}]\hfill\\ 
For a function $f:\Omega \rightarrow \mathbb{R}$, the spatial average is denoted by (recall $|\Omega|=1$)
\begin{equation*}
													\overline{f} = \int_{\Omega}f(x)\,dx.
\end{equation*}
Moreover, for a quantity denoted by small letters, we introduce the short hand notation of the same uppercase letter as its square root, 
e.g. 
$$
C_i = \sqrt{c_i}, \quad \text{and}\quad C_{i,\infty} = \sqrt{c_{i,\infty}}.
$$

\item[{\it Additivity of Entropy}] see e.g. \cite{DeFe08,DFIFIP}, \cite[Lemma 2.3]{MHM}
\begin{equation}\label{add}
													\begin{aligned}
																\mathcal{E}(\cc) - \mathcal{E}(\cc_{\infty}) &= (\mathcal{E}(\cc) - \mathcal{E}(\overline{\cc}))+(\mathcal{E}(\overline{\cc}) - \mathcal{E}(\cc_{\infty}))\\
																&=\sum_{i=1}^{I}\int_{\Omega}c_i\log{\frac{c_i}{\overline{c}_i}}dx + \sum_{i=1}^{I}\left(\overline{c}_i\log{\frac{\overline{c}_i}{c_{i,\infty}}} - \overline{c}_i + c_{i,\infty}\right).
													\end{aligned}
										\end{equation}
							\item[\it The Poincar\'{e} inequality]
For all $f\in H^1(\Omega)$, there exists $C_P(\Omega)>0$ depending only on $\Omega$ such that 
								\begin{equation*}
									\|\nabla f\|^2 \geq C_P\|f - \overline{f}\|^2.
								\end{equation*}
							\item[\it The Logarithmic Sobolev inequality]
							There exists $C_{LSI}>0$ depending only on $\Omega$ such that 
							\begin{equation}\label{LSI}
								\int_{\Omega}\frac{|\nabla f|^2}{f}dx \geq C_{LSI}\int_{\Omega}f\log{\frac{f}{\overline{f}}}dx.
							\end{equation}
\item[{\it An elementary inequality}]\ 					
$
(a-b)(\log a - \log b) \geq 4\bigl(\sqrt{a} - \sqrt{b}\bigr)^2.\smallskip
$
							\item[{\it An elementary function}]\hfill\\
										Consider $\Phi: [0,+\infty) \rightarrow [0,+\infty)$ defined as (and continuously extended at $z=0,1$)
										\begin{equation}\label{Phi}
													\Phi(z) = \frac{z\log z - z + 1}{\left(\sqrt{z} - 1\right)^2}.
										\end{equation}
Then, $\Phi$ is increasing with $\lim\limits_{z\rightarrow 0}\Phi(z) = 1$ and $\lim\limits_{z\rightarrow 1}\Phi(z) = 2$.												  
				\end{description}


\subsection{Preliminary estimates and inequalities}\hfill

\medskip
{\color{black} The mass conservation laws \eqref{g9} play a crucial role in the analysis of this paper. However, due to the very low regularity of renormalised solutions, a rigorous justification of the formal conservation laws is not trivial. Following the ideas from \cite[Proposition 6]{Fis16}, which proved that any renormalised solution satisfies the weak entropy entropy-dissipation law \eqref{weak_entropy}, we prove here that the conservation laws \eqref{g9} hold indeed for a large class of renormalised solutions, which includes all renormalised solutions of the systems considered in this paper. Note that the following proposition is formulated for general systems \eqref{SS} with an additional assumption on the matrix $\mathbb Q$, which is satisfied by the systems \eqref{SysSingle} and \eqref{SysChain}.
\begin{proposition}[Mass conservation for renormalised solutions]\label{mass_renormalised}
Any renormalised solution $\cc$ of \eqref{SS}--\eqref{kr}, 
\textcolor{black}{for which the non-unique matrix $\mathbb Q$
as introduced in \eqref{Q} can be chosen to have non-negative entries, satisfies all the associated mass conservation laws \eqref{g9}, i.e. 
we have for all $t\geq 0$
}
\begin{equation*}
		\mathbb Q\,\overline{\cc}(t) = \mathbb Q\,\overline{\cc_0}.
\end{equation*}
\end{proposition}
\begin{remark}\label{massrenorm}
\textcolor{black}{ 
For detailed balance systems, which do not allow to choose a non-negative version  of the matrix $\mathbb Q$, it is still possible to show that renormalised 
solutions, which are constructed via an approximation scheme as in \cite{Fi14}, satisfy also the mass conservation laws \eqref{g9}. This follows directly from passing 
to the limit in the approximative conservation laws 
and the regularity of renormalised solutions in $L^{\infty}_{loc}([0,+\infty); (L\log L)(\Omega))$. 
However, it is unclear whether all renormalised solutions defined via \eqref{renormalised} 
in Proposition \ref{renormalised_sol} can indeed be constructed via 
such an approximation scheme. In \cite{Fis16}, Fisher proved that all renormalised solutions defined via \eqref{renormalised} satisfy the weak entropy entropy-dissipation inequality \eqref{weak_entropy}. Here we provide a similar proof that all renormalised solutions defined via \eqref{renormalised}
satisfy any mass conservation law of the form \eqref{g9}, provided that the corresponding row of 
 the matrix $\mathbb Q$ consists of non-negative entries. 
 Thus, for any system of the form \eqref{SS}--\eqref{kr}, where the non-unique matrix $\mathbb Q$ can be chosen to 
 only consist of non-negative entries, all renormalised solutions 
 satisfy all conservation laws. \\
 Note that there are some detailed balance systems like $\mathcal{C}_1 \rightleftharpoons 2\mathcal{C}_1 + \mathcal{C}_2$, where $\mathbb Q=(1,-1)$ and it is thus impossible to chose a matrix $\mathbb Q$ with non-negative entries. However, this example constitutes a chemically meaningless reaction, and one could speculate that the existence of non-negative versions of matrices $\mathbb Q$ is related to ``meaningful" chemical reactions. 
However, we are unaware of any result in the direction of classifying chemical reactions according to non-negative matrices $\mathbb Q$.
} 
\end{remark}
\begin{proof}[Proof of Proposition \ref{mass_renormalised}]
	For each $M>0$, we choose a smooth function $\theta_M: [0,\infty) \rightarrow [0,1]$ such that $\theta_M(s) = s$ for $s\in [0,M]$ and $\theta_M'(s) = 0$ for $s\geq 2M$, and moreover $\theta''_M(s) \leq \mathbf{1}_{[M,2M]}\frac{1}{Ms}$ for all $s\geq 0$. For any $k=1,\ldots, m$, let $(q_{k\ell})_{\ell=1,\ldots, I}$ be the $k$-th row of $\mathbb Q$. Then, by choosing
	\begin{equation*}
		\xi(\cc) = \theta_M\biggl(\sum_{\ell=1}^{I}q_{k\ell}c_{\ell}\biggr)
	\end{equation*}
	and $\psi = 1$ in \eqref{renormalised}, we obtain
	\begin{multline}\label{f1}
		\int_{\Omega}\theta_{M}\biggl(\sum_{\ell=1}^{I}q_{k\ell}c_{\ell}(T)\biggr)dx - \int_{\Omega}\theta_{M}\biggl(\sum_{\ell=1}^{I}q_{k\ell}c_{\ell,0}\biggr)dx\\
		= -\sum_{i,j=1}^{I}\int_{0}^{T}\!\!\int_{\Omega}q_{ki}q_{kj}\theta_M''\biggl(\sum_{\ell=1}^{I}q_{k\ell}c_{\ell}\biggr) d_i\nabla c_i\cdot \nabla c_jdxdt - \int_{0}^{T}\!\!\int_{\Omega}\theta_M'\biggl(\sum_{\ell=1}^{I}q_{k\ell}c_{\ell}\biggr)\biggl(\sum_{i=1}^{I}q_{ki}R_i(\cc)\biggr)dxdt.
	\end{multline}
	Due to the definition of $\mathbb Q$ (see \eqref{Q}), we have for the last term that 
	\begin{equation*}
		\sum_{i=1}^{I}q_{ki}R_i(\cc) = 0
	\end{equation*}
for all $\cc\in \mathbb R^{I}_{+}$. It remains to control the first term on the right hand side of \eqref{f1}. By recalling that  renormalised solution $\cc$ satisfies $\sqrt{c_i}\in L^2(0,T; H^1(\Omega))$ (see Proposition \ref{renormalised_sol}), 
we  estimate with $\theta_M''(s) \leq \frac{1}{Ms}$
\begin{equation*}
		\begin{aligned}
		&\left|\int_0^T\!\!\!\int_{\Omega}q_{ki}q_{kj}\theta_M''\biggl(\sum_{\ell=1}^{I}q_{k\ell} c_{\ell}\biggr)\nabla c_i\cdot \nabla c_j dxdt \right|
		\leq 4\int_0^T\!\!\!\int_{\Omega}\left|\theta_M''\biggl(\sum_{\ell=1}^{I}q_{k\ell} c_{\ell}\biggr)q_{ki}\sqrt{c_i}\,q_{kj}\sqrt{c_j}\right||\nabla \sqrt{c_i}||\nabla \sqrt{c_j}|dxdt\\
		&\leq \frac{4}{M}\int_0^T\!\!\!\int_{\Omega}\left|\frac{q_{ki}\sqrt{c_i}\,q_{kj}\sqrt{c_j}}{\sum_{\ell=1}^{I}q_{k\ell} c_{\ell}}\right||\nabla \sqrt{c_i}||\nabla \sqrt{c_j}|dxdt
		\leq \frac{C}{M}\|\sqrt{c_i}\|_{L^2(0,T; H^1(\Omega))}\|\sqrt{c_j}\|_{L^2(0,T;H^1(\Omega))}
		\end{aligned}
\end{equation*}
where we have used that $\frac{q_{ki}\sqrt{c_i}q_{kj}\sqrt{c_j}}{\sum_{\ell=1}^{I}q_{k\ell} c_{\ell}}$ is bounded for 
any non-negative vector $(q_{k\ell})_{\ell=1,\ldots, I}$. 
Inserting this into \eqref{f1} and then letting $M\rightarrow\infty$, noting that $c_i\in L^{\infty}(0,T;L^1(\Omega))$, we obtain for all $T>0$
	\begin{equation*}
		\sum_{\ell=1}^{I}q_{k\ell}\overline{c_{\ell}(T)} = \sum_{\ell=1}^{I}q_{k\ell}\overline{c_{\ell,0}}.
	\end{equation*}
	Since $k = 1,\ldots, m$ is arbitrary, the proof of Lemma \ref{mass_renormalised} is complete.
\end{proof}
}

\begin{lemma}[$L^1$-bounds]\label{bounded}
Assume that the initial data $\cc_0$ are nonnegative and satisfies $\mathcal{E}(\cc_0) < +\infty$. 
Then, {the non-negative renormalised solutions to \eqref{SS}--\eqref{kr} satisfy}
\begin{equation*}
		\|c_i(t)\|_{L^1(\Omega)} \leq K:= 2(\mathcal{E}({\cc_0}) + I)\qquad \forall t>0, \quad \forall i=1,2,\ldots, I.
		\end{equation*}
\end{lemma}
\begin{proof}

{\color{black}
Thanks to \cite[Proposition 6]{Fis16}, any renormalised solution to \eqref{SS} satisfies the weak entropy entropy-dissipation law \eqref{weak_entropy}, and consequently
\begin{equation}
	\mathcal{E}(\cc)(t) \leq \mathcal{E}(\cc_0) \qquad \text{ for all } t>0
\end{equation}
or equivalently
\begin{equation*}
\sum_{i=1}^{I}\int_{\Omega}\left(c_i(x,t)\log c_i(x,t) - c_i(x,t) + 1\right)dx \leq \mathcal{E}({\cc_0}) \qquad \text{ for all } t>0.
\end{equation*}
By using the elementary inequalities $x\log x - x + 1 \geq \left(\sqrt{x} - 1\right)^2 \geq \frac{1}{2}x - 1$ for all $x\ge 0$, we get
\begin{equation*}
\frac{1}{2}\sum_{i=1}^{I}\int_{\Omega}c_i(x,t)\,dx \leq \mathcal{E}({\cc_0}) + I.
\end{equation*}
This, combined with the non-negativity of solutions, completes the proof of the Lemma.
}
\end{proof}
The following Csisz\'{a}r-Kullback-Pinsker type inequality shows that $L^1(\Omega)$-convergence to equilibrium  follows from convergence of solutions in relative entropy $\mathcal{E}(\cc) - \mathcal{E}(\cc_{\infty})$. For a generalised Csisz\'{a}r-Kullback-Pinsker inequality, we refer to \cite{Arnold}. Here, we give an elementary proof using only the natural bounds given in Lemma \ref{bounded}.
\begin{lemma}[Csisz\'{a}r-Kullback-Pinsker type inequality]\label{CKP}
							For all $\cc\in L^1(\Omega; [0,+\infty)^{I})$ such that $\mathbb{Q}\,\overline{\cc} = \mathbb{Q}\,\cc_{\infty}$ and $\overline{c}_i \leq K$ for all $i=1,2,\ldots, I$ with some $K>{\max_{i}\{c_{i,\infty}\}>0}$, we have
\begin{equation*}
\mathcal{E}(\cc) - \mathcal{E}(\cc_{\infty}) \geq C_{CKP}\sum_{i=1}^{I}\|c_i - c_{i,\infty}\|_{L^1(\Omega)}^2
\end{equation*}							
where the constant $C_{CKP}$ depends only on the domain $\Omega$ and the constant $K$.
\end{lemma}
\begin{proof}
By recalling $|\Omega|=1$ and using the additivity of the relative entropy, we have
\begin{equation}\label{q1}								\mathcal{E}(\cc) - \mathcal{E}(\cc_{\infty}) =\sum_{i=1}^{I}\int_{\Omega}c_i\log{\frac{c_i}{\overline{c}_i}}dx + \sum_{i=1}^{I}\left(\overline{c}_i\log{\frac{\overline{c}_i}{c_{i,\infty}}} - \overline{c}_i + c_{i,\infty}\right).
\end{equation}
By using the classical Csisz\'{a}r-Kullback-Pinsker inequality, we have
\begin{equation}\label{q2}
\int_{\Omega}c_i\log{\frac{c_i}{\overline{c}_i}}dx \geq C_0\|c_i  - \overline{c}_i\|_{L^1(\Omega)}^2
\end{equation}
for all $i=1,2,\ldots, I$, where the constant $C_0$ depends only on the domain $\Omega$. On the other hand, by applying the elementary inequality $x\log(x/y) - x + y \geq (\sqrt{x} - \sqrt{y})^2$, we obtain
\begin{equation}\label{q3}
\begin{aligned}
\mathcal{E}(\overline{\cc}) - \mathcal{E}(\cc_{\infty}) &= \sum_{i=1}^{I}\left(\overline{c}_i\log{\frac{\overline{c}_i}{c_{i,\infty}}} - \overline{c}_i + c_{i,\infty}\right)
\geq \sum_{i=1}^{I}\left(\sqrt{\overline{c}_i} - \sqrt{c_{i,\infty}}\right)^2
= \sum_{i=1}^{I}\frac{(\overline{c}_i - c_{i,\infty})^2}{\left(\sqrt{\overline{c}_i} + \sqrt{c_{i,\infty}}\right)^2}\\
	&\geq \frac{1}{4K}\sum_{i=1}^{I}(\overline{c}_i - c_{i,\infty})^2,
	\end{aligned}
\end{equation}
where $K$ is given in Lemma \ref{bounded}.
By combining \eqref{q1}--\eqref{q3}, we obtain
\begin{equation*}
\begin{aligned}
\mathcal{E}(\cc) - \mathcal{E}(\cc_{\infty}) &\geq C_0\sum_{i=1}^{I}\|c_i  - \overline{c}_i\|_{L^1(\Omega)}^2 + \frac{1}{4K}\sum_{i=1}^{I}(\overline{c}_i - c_{i,\infty})^2\\
													&\geq \min\{C_0; 1/4K\}\sum_{i=1}^{I}\left(\|c_i  - \overline{c}_i\|_{L^1(\Omega)}^2 + \|\overline{c}_i - c_{i,\infty}\|_{L^1(\Omega)}^2\right)\\
&\geq \frac{1}{2}\min\{C_0; 1/4K\}\sum_{i=1}^{I}\|c_i - c_{i,\infty}\|_{L^1(\Omega)}^2,
\end{aligned}
\end{equation*}
 which is the desired inequality with $C_{CKP} = \frac{1}{2}\min\left\{C_0; \frac{1}{4K}\right\}$.
\end{proof}
The following entropy entropy-dissipation estimate was established in \cite{MHM}.
\begin{theorem}\cite{MHM}\label{MHM}
Assume that \eqref{SS}--\eqref{R} satisfies the assumption $(\mathbf{A1})$ and $(\mathbf{A2})$. 

Then, for a given positive initial mass vector $\M\in \mathbb{R}^m_+$ and for all $\cc\in L^1(\Omega; [0,+\infty)^{I})$ satisfying $\mathbb{Q}\,\overline{\cc} = \M$, there exists 
a positive constant $\lambda>0$ such that
	\begin{equation*}
	\mathcal{D}(\cc) \geq \lambda(\mathcal{E}(\cc) - \mathcal{E}(\cc_{\infty}))
	\end{equation*} 
with $\mathcal E (\cc)$ and $\mathcal D(\cc)$ given in \eqref{entropy} and \eqref{entropy_dissipation}, respectively, and where $\cc_{\infty}$ is the detailed balance equilibrium of \eqref{SS} corresponding to $\M$.
\end{theorem}
We emphasise that, though this Theorem gives the existence of $\lambda>0$, it seems difficult to extract an explicit estimate of $\lambda$ except in some special cases, e.g. the quadratic system arising from the reaction $2\mathcal{C}_1 \leftrightharpoons \mathcal{C}_2$. The main reason is the (highly elegant) convexification argument used in \cite{MHM}, which seems very difficult (if not impossible) to make explicit for general systems.
\medskip
				
In this paper, we propose a constructive method to prove an EED estimate. The method applies elementary estimates and has the advantage of allowing to explicitly estimate the rates and constants of convergence to equilibrium.
Before detailing our approach, we shall further remark on  assumption $(\mathbf{A2})$ on the absence of boundary equilibria.
				
\begin{remark}[Boundary equilibrium]\label{boundequi}
The validity of Theorem \ref{MHM} fails if the system \eqref{SS} has a boundary equilibrium. As example, consider the single reversible reaction $2\mathcal{A} \leftrightharpoons \mathcal{A}+\mathcal{B}$ with normalised reaction rate constants $k_f = k_b = 1$, which leads to the following system
\begin{equation*}\label{2x2}
\begin{cases}
a_t - d_a\Delta a = -a^2 + ab, &\quad x\in\Omega, \quad t>0,\\
b_t - d_b\Delta b = a^2 - ab, &\quad x\in\Omega, \quad t>0,\\
\nabla a \cdot \nu = \nabla b\cdot \nu = 0, &\quad x\in\partial\Omega, \quad t>0,\\
a(x,0) = a_0(x), \; b(x,0) = b_0(x), &\quad x\in\Omega,
\end{cases}
\end{equation*}
and where we assume that $d_a, d_b>0$. This system has one mass conservation law
	\begin{equation*}
	\int_{\Omega}(a(x,t) + b(x,t))dx = \int_{\Omega}(a_0(x) + b_0(x))dx =: \M>0 \qquad \forall t>0.
\end{equation*}
It is easy to see that the system possesses the positive detailed balance equilibrium $(a_{\infty}, b_{\infty}) = \left(\frac{\M}{2}, \frac{\M}{2}\right)$ and the boundary equilibrium $(a^*_{\infty}, b^*_{\infty}) = (0, \M)$. 
Moreover, we have the entropy functional
	\begin{equation*}
	\mathcal{E}(a,b) = \int_{\Omega}(a\log a - a + 1)dx + \int_{\Omega}(b\log b - b + 1)dx
	\end{equation*}
and the entropy-dissipation functional
\begin{equation*}
\mathcal{D}(a,b) = \int_{\Omega}d_a\frac{|\nabla a|^2}{a}dx + \int_{\Omega}d_b\frac{|\nabla b|^2}{b}dx + \int_{\Omega}a(a-b)(\log a - \log b)dx.
\end{equation*}
By defining $Z = \{(a,b)\in \mathbb{R}_+^2: a + b = \M\}$, we can easily compute that
\begin{equation*}
	\lim\limits_{Z \ni(a, b)\rightarrow (a^*_{\infty}, b^*_{\infty})}\mathcal D(a,b) = 0
\end{equation*}
	and
\begin{equation*}
\lim\limits_{Z\ni (a,b)\rightarrow (a^*_{\infty}, b^*_{\infty})}(\mathcal{E}(a,b) - \mathcal{E}(a_{\infty}, b_{\infty})) = \M\log 2 >0.
\end{equation*}
As a consequence, {\normalfont there does not exist a constant $\lambda>0$} such that
\begin{equation*}
\mathcal{D}(a,b) \geq \lambda(\mathcal{E}(a,b) - \mathcal{E}(a_{\infty}, b_{\infty}))
\end{equation*}
for all functions $a, b: \Omega \rightarrow \mathbb{R}_+$ satisfying $\int_{\Omega}(a(x) + b(x))dx = \M$.

\medskip							
Therefore, if \eqref{SS}--\eqref{R} has a boundary equilibrium then we cannot expect in general global exponential convergence to equilibrium. We can only expect exponential convergence of trajectories, which are uniformly bounded away from the boundary equilibrium.
Interestingly, it is conjectured in the case of ODE reaction systems that for complex balanced reaction systems, even if the system possesses boundary equilibria, a trajectory starting from any other positive initial state will always converge to the unique positive equilibrium as time goes to infinity. This conjecture is given the name {\normalfont Global Attractor Conjecture} in \cite{CrDietal09} and is considered to be one of the most important open problems in the theory of chemical reaction networks. A recent proposed proof of the conjecture in ODE setting is under verification \cite{Cra15}.

For PDE systems featuring boundary equilibria, we refer the reader to \cite{DFT16,GH97}.
\end{remark}
				
\subsection{A constructive method to prove the EED estimate}\label{subsec:2.2}
Although Theorem \ref{MHM} provides the existence of $\lambda>0$ satisfying the entropy entropy-dissipation estimate, it does not seem to allow for explicit estimates on $\lambda$ when the reaction network has more than two substances, for example,
\begin{equation*}
\alpha \mathcal{C}_1 + \beta \mathcal{C}_2 \leftrightharpoons \gamma \mathcal{C}_3 \qquad \text{ or } \qquad \mathcal{C}_1 + \mathcal{C}_2 \leftrightharpoons \mathcal{C}_3 + \mathcal{C}_4.
	\end{equation*}
Inspired by \cite{DeFe08, DFIFIP, BFL,FL}, we propose a general approach to prove an entropy entropy-dissipation estimate based on the mass conservation laws of a reaction-diffusion system of the form \eqref{SS}--\eqref{R}, which allows explicit estimates of the rates and constants of convergence. 

By recalling the key EED estimate
				\begin{equation}\label{MainE1}
							\mathcal{D}(\cc) \geq \lambda(\mathcal{E}(\cc) - \mathcal{E}(\cc_{\infty})),
\end{equation}
with $\mathcal{E}(\cc)$ and $\mathcal{D}(\cc)$ are given in \eqref{entropy} and \eqref{entropy_dissipation}, we point out that the right hand side is zero if and only if $\cc \equiv \cc_{\infty}$, while the left hand side is zero for all constant states $\cc^*\in (0,+\infty)^{I}$ satisfying $(\cc^*)^{\boldsymbol{\alpha}^r} = (\cc^*)^{\boldsymbol{\beta}^r}, \; \forall r=1,2,\ldots, R$ and that $\cc^*$ identifies with $\cc_{\infty}$ if and only if $\mathbb{Q}\,\cc^* = \M$. Hence, the EED estimate 
\eqref{MainE1} has to crucially take into account all the conservations laws of the system.
\medskip

\begin{remark}[Explicit constants]\label{rem:explicit}
We remark that the rates obtained by our approach are not optimal. Therefore, for the sake of readability, we shall denote at some places by $K_i$ constants, which can be estimated explicitly, but for which we don't state unnecessary long expressions.
\end{remark}

The aim of this subsection is to prove Theorem \ref{the:main}. Once this theorem is proved, it remains only to prove the finite dimensional inequality \eqref{ODEEqui} and the lower bound for entropy dissipation \eqref{FarEqui} in order to apply Theorem 
\ref{the:main} in the proofs of the Theorems \ref{mainsingle} and \ref{mainChain} as presented in Sections \ref{sec:3} and \ref{sec:4}. 
There, the proofs of the inequalities \eqref{ODEEqui} and \eqref{FarEqui} are crucially based on the structure of the mass conservation laws 
\eqref{MassSingle} and \eqref{MassChain}, respectively. 

\bigskip

We first prove the following intermediate lemma.
\begin{lemma}\label{lem:intermediate}
Consider the general mass-action law reaction-diffusion system
\eqref{SS}--\eqref{R} with the associated entropy 
and entropy-dissipation functionals \eqref{entropy} and \eqref{entropy_dissipation}.

Assume that there exists a constant $H_6$ such that 
	\begin{equation}\label{intermediate} 
		\sum_{i=1}^{I}\|\nabla C_i\|^2 + \sum_{r=1}^{R}\left(\overline{\mathbf C}^{\boldsymbol{\alpha} ^r} - \overline{\mathbf C}^{\boldsymbol{\beta} ^r}\right)^2
		\geq H_6\sum_{i=1}^{I}\left(\sqrt{\overline{C_i^2}} - C_{i,\infty}\right)^2,
\end{equation}
where we recall the short hand notation $C_i = \sqrt{c_i}$, $C_{i,\infty} = \sqrt{c_{i,\infty}}$, e.t.c.

Then,{ for all non-negative $\cc\in L^1(\Omega; [0,+\infty)^{I})$ with $\overline{c_i} \leq K$}, there holds
\begin{equation*}
\mathcal D(\cc) \geq \lambda(\mathcal{E}(\cc) - \mathcal{E}(\ww))
\end{equation*}
for $\mathcal{E}(\cc)$ and $\mathcal{D}(\cc)$ as given in \eqref{entropy} and \eqref{entropy_dissipation} and
for a constant $\lambda >0$ depending {\normalfont{explicitly}} on $H_6$ and $K$,  the domain $\Omega$, the diffusion coefficients $d_i$, the stoichiometric coefficients $\aa^r, \bb^r$, the reaction rate constants $k_{r}$, and the initial mass vector $\M$.
\end{lemma}
\begin{proof}
The proof of this lemma is devided into three steps. 
\begin{description} 
\item[{\bf Step 1 (Use of the Logarithmic Sobolev Inequality)}]\hfill\\
The idea of this step is to divide the relative entropy $\mathcal{E}(\cc) - \mathcal{E}(\cc_{\infty})$ into two parts
by using the additivity properties of the entropy \eqref{add}:
\begin{equation}\label{add2}
\begin{aligned}
\mathcal{E}(\cc) - \mathcal{E}(\cc_{\infty})&= \left(\mathcal{E}(\cc) - \mathcal{E}(\overline{\cc})\right) + \left(\mathcal{E}(\overline{\cc}) - \mathcal{E}(\cc_{\infty})\right)\\
&= \sum_{i=1}^{I}\int_{\Omega}c_i\log{\frac{c_i}{\overline{c_i}}}dx + \sum_{i=1}^{I}\left(\overline{c_i}\log{\frac{\overline{c_i}}{c_{i,\infty}}} - \overline{c_i} + c_{i,\infty}\right),\\
\end{aligned}
\end{equation}
where the first part is controlled in terms (of a part of) the Fisher information in the entropy dissipation via the lower bounds $d_i \geq d_{min}$ and the Logarithmic Sobolev Inequality, i.e. 
\begin{equation*}
\int_{\Omega}d_i\frac{|\nabla c_i|^2}{c_i}dx \geq C_{LSI} \,d_{min}\int_{\Omega}c_i\log{\frac{c_i}{\overline{c_i}}}dx, \qquad 
\forall i=1,\ldots,I
\end{equation*}
and the second term on the right hand side of \eqref{add2} depends only on spatial averages of concentrations,
which have the two advantages of obeying the conservation laws and of satisfying the natural uniform a-priori bounds of Lemma \ref{bounded}. 
										
In particular, the Logarithmic Sobolev Inequality allows to estimate 
\begin{equation}\label{step1}
\frac{1}{2}\mathcal{D}(\cc) \geq \frac{1}{2}C_{LSI}\,d_{min}(\mathcal{E}(\cc) - \mathcal{E}(\overline{\cc})).
\end{equation}
\vskip .5cm
\item[{\bf Step 2 (Transformation into quadratic terms for the square roots of the concentrations)}]\hfill\\
We estimate (the remaining part of) $\mathcal{D}(\cc)$ as well as $\mathcal{E}(\ovv{\cc}) - \mathcal{E}(\cc_{\infty})$ in terms of $L^2$-distance of the square roots $C_i$ of the concentrations $c_i$.
The associated $L^2$-norms are significantly easier to handle than logarithmic terms. For $\mathcal{D}(\cc)$ we estimate
\begin{align}
\frac 12\mathcal{D}(\cc) &= \frac 12\sum_{i=1}^{I}\int_{\Omega}d_i\frac{|\nabla c_i|^2}{c_i}dx + \frac 12\sum_{r=1}^{R}k^r\int_{\Omega}(\cc^{\boldsymbol{\alpha} ^r} - \cc^{\boldsymbol{\beta} ^r})(\log \cc^{\boldsymbol{\alpha} ^r} - \log \cc^{\boldsymbol{\beta} ^r})dx\nonumber\\
&\geq 2d_{min}\sum_{i=1}^{I}\|\nabla C_i\|^2 + 2\sum_{r=1}^{R}k^r\left\|\mathbf C^{\boldsymbol{\alpha} ^r} - \mathbf C^{\boldsymbol{\beta} ^r}\right\|^2\nonumber \\
&\geq K_1\left(\sum_{i=1}^{I}\|\nabla C_{i}\|^2 + \sum_{r=1}^{R}\left\|\mathbf C^{\aa^r} - \mathbf C^{\bb^r}\right\|^2\right),
\label{g16}
\end{align}
with
\begin{equation}\label{K1}
	K_1 = 2\min\{d_{min}, \min_{r=1,\ldots, R}\{k^r\}\},
\end{equation}
and we recall the abbreviations $C_i = \sqrt{c_i}$, $\cc = (C_1,C_2,\ldots, C_I)^{\top}$ and the elementary inequality $(a-b)(\log a - \log b) \geq 4(\sqrt{a} - \sqrt{b})^2$. 
										
For the second terms on the right hand side of \eqref{add2}, we use the function $\Phi$ in \eqref{Phi} to estimate
\begin{equation}\label{g17}
\begin{aligned}
\mathcal{E}(\overline{\cc}) - \mathcal{E}(\cc_{\infty}) &= \sum_{i=1}^{I}\left(\overline{c_i}\log{\frac{\overline{c_i}}{c_{i,\infty}}} - \overline{c_i} + c_{i,\infty}\right)
= \sum_{i=1}^{I}\Phi\left(\frac{\overline{c_i}}{c_{i,\infty}}\right)\left(\sqrt{\overline{c_i}} - \sqrt{c_{i,\infty}}\right)^2\\
&\leq K_2\sum_{i=1}^{I}\left(\sqrt{\overline{C_i^2}} - C_{i,\infty}\right)^2
\end{aligned}
\end{equation}
with
\begin{equation}\label{K2}
K_2 = \max_{i=1,\ldots, I}\left\{\Phi\left(\frac{K}{c_{i,\infty}}\right)\right\},
\end{equation}
where we have used Lemma \ref{bounded} to estimate $\overline{c_i} \leq K$ for all $i=1,\ldots, I$, and the monotonicity of the function $\Phi$. 

\vskip .5cm

\item[{\bf Step 3 (Control of reaction dissipation term by a reaction dissipation term for averages)}]
The estimate \eqref{g16} is a proper lower bound of the 
entropy dissipation $\mathcal{D}(\cc)$ in the sense 
that all the conservation laws are required in order to see that the right hand side of \eqref{g16} is zero only at the equilibrium $\CC_{\infty}$. In this step, we continue  to estimate below the right hand side of \eqref{g16} in terms of the left hand side of \eqref{intermediate}.


Indeed, the two terms on right hand side of \eqref{g16} represents 
lower bounds for the entropy-dissipation caused by all the diffusion and reaction processes of the considered system. In order to be able to use the constraints provided by the conservation laws, we shall bound the reaction dissipation term  below by a reaction dissipation term for spatially averaged concentrations. More precisely, by denoting $\overline{\mathbf C} = (\overline{C_1}, \ldots, \overline{C_I})^{\top}$, we have
\begin{lemma}[Reaction dissipation terms for averaged concentrations]\label{average}\hfill\\
{Under the assumptions of Lemma \ref{lem:intermediate}}, there exists an explicit constant $K_3>0$ such that
		\begin{equation}\label{g19}
			\quad 2\|\nabla C_i\|^2 + 2\sum_{r=1}^{R}\left\|\mathbf C^{\boldsymbol{\alpha} ^r} - \mathbf C^{\boldsymbol{\beta} ^r}\right\|^2
			\geq K_3\Bigl(\sum_{i=1}^{I}\left\|\nabla C_i\right\|^2 + \sum_{r=1}^{R}\left(\overline{\mathbf C}^{\boldsymbol{\alpha} ^r} - \overline{\mathbf C}^{\boldsymbol{\beta} ^r}\right)^2\Bigr).
		\end{equation}
\end{lemma}
\noindent 
It's worth noticing that comparing to a related estimate in \cite{DeFe08}, where the considered quadratic nonlinearities allowed to exploit certain $L^2$-orthogonality structures,  Lemma \ref{average} is significantly more complicated due to the arbitrary order of the considered nonlinearities. In the proof of Lemma \ref{average}, we shall introduce some new ideas, which are motivated by \cite{FL} and consist of a domain decomposition 
to overcome the difficulties caused by the general nonlinearities. This idea is also applicable to volume-surface reaction-diffusion systems, see \cite{BFL}. 
\begin{proof}[Proof of Lemma \ref{average}]
We first prove that 
						\begin{equation} \label{g19'}
									\sum_{i=1}^{I}\|\nabla C_i\|^2 + 2\sum_{r=1}^{R}\left\|\mathbf C^{\boldsymbol{\alpha} ^r} - \mathbf C^{\boldsymbol{\beta} ^r}\right\|^2\\
									\geq \kappa\sum_{r=1}^{R}\left(\overline{\mathbf C}^{\boldsymbol{\alpha} ^r} - \overline{\mathbf C}^{\boldsymbol{\beta} ^r}\right)^2
						\end{equation}
for an explicit constant $\kappa>0$. Then, \eqref{g19} follows by choosing 
$$
K_3 = \min\{1, \kappa\}.
$$
The proof of \eqref{g19'} introduces pointwise deviations of the concentrations around their spatial averages, which are defined as follows: for all $i = 1, 2, \ldots, I$, we set
\begin{equation*}\label{deviation}
	\delta_i(x) = C_i(x) - \overline{C}_i,\qquad \text{ for } x\in\Omega.
\end{equation*}
Thanks to the non-negativity of $C_i$ and Lemma \ref{bounded}, we have that $\delta_i \in [-\sqrt{K}; +\infty)$. Fixing a constant $L>\sqrt{K}>0$, we can decompose $\Omega$ as
							\begin{equation*}\label{DecomposeOmega}
										\Omega = S \cup S^{\perp},
							\end{equation*}
							where
							\begin{equation*}\label{S}
										S = \{x\in\Omega: |\delta_i(x)| \leq L, \quad \forall i = 1,2, \ldots, I\}.
							\end{equation*} 
We will prove \eqref{g19'} on both $S$ and $S^{\perp}$. On $S$ we estimate for a $\gamma\in (0,2)$
by using Taylor expansion (i.e. $\prod_{i=1}^{I}\left(\overline{C}_i+\delta_i\right)^{\alpha^r_i} = \overline{\mathbf C}^{\boldsymbol{\alpha}^r} +\sum_{i=1}^{I} R_i( \overline{\mathbf C},\aa^r, \mathbf{\delta})\delta_i$ with bounded remainders $|R_i( \overline{\mathbf C},\aa^r, \mathbf{\delta})|\le C(K,\boldsymbol{\alpha}^r,L)$) and Young's inequality that
\begin{align}\label{e14}										
\gamma\sum_{r=1}^{R}\left\|\mathbf C^{\boldsymbol{\alpha}^r} - \mathbf C^{\boldsymbol{\beta}^r}\right\|^2_{L^2(S)}
&= \gamma\sum_{r=1}^{R}\left\|\prod\limits_{i=1}^{I}\left(\overline{C}_i+\delta_i\right)^{\alpha^r_i} - \prod\limits_{i=1}^{I}\left(\overline{C}_i+\delta_i\right)^{\beta^r_i}\right\|^2_{L^2(S)}\\
&\geq \frac 12 \gamma\sum_{r=1}^{R}\left[\overline{\mathbf C}^{\boldsymbol{\alpha}^r} - \overline{\mathbf C}^{\boldsymbol{\beta}^r}\right]^2|S|
- \gamma\, C(L, K, \aa^r, \bb^r)\sum_{i=1}^{I}\|\delta_i\|^2_{L^2(S)},
\end{align}
where $C(L, K, \aa^r, \bb^r)$ is an explicit constant which {\it does not} depend on $S$. On the other hand, by using the Poincar$\acute{\mathrm{e}}$ inequality,
							\begin{equation*}\label{e15}
								\|\nabla f\|^2 \geq C_P\|f - \overline{f}\|^2 \geq C_P\|f - \overline{f}\|_{L^2(S)}^2,
							\end{equation*}
							we have
							\begin{equation}\label{e16}
										\begin{aligned}
													\frac 12\sum_{i=1}^{I}\|\nabla C_i\|^2 \geq \frac 12C_P\sum_{i=1}^{I}\|\delta_i\|^2 \geq \frac 12C_P\sum_{i=1}^{I}\|\delta_i\|^2_{L^2(S)}.
										\end{aligned}
							\end{equation}
							From \eqref{e14} and \eqref{e16}, if we choose $\gamma \in (0,2)$ such that $2\gamma C(L, K, \aa^r, \bb^r) \leq C_{P}$, then we have
							\begin{equation}\label{e17}
										\frac 12\sum_{i=1}^{I}\|\nabla C_i\|^2 + 2\sum_{r=1}^{R}\left\|\mathbf C^{\boldsymbol{\alpha}^r} - \mathbf C^{\boldsymbol{\beta}^r}\right\|^2_{L^2(S)} 
										\geq \frac 12 \gamma\sum_{r=1}^{R}\left[\overline{\mathbf C}^{\boldsymbol{\alpha}^r} - \overline{\mathbf C}^{\boldsymbol{\beta}^r}\right]^2|S|.
							\end{equation}
To estimate \eqref{g19'} on $S^{\perp}$, we note that by the definition of $S^{\perp}$ and $L$, we have
\begin{equation*}\label{e18}
S^{\perp} = \{x\in\Omega: \quad \delta_i(x) > L \ \text{ for some } \ i = 1,2 \ldots, I\}.
\end{equation*}
							Hence,
							\begin{equation}\label{e19}
										\begin{aligned}
													|S^{\perp}| &\le \sum_{i=1}^{I}\left|\{x\in \Omega: \delta_i(x)>L\}\right|\le \sum_{i=1}^{I}\left|\{x\in\Omega: \delta^2_i(x) > L^2\}\right|\\
													&\leq \frac{1}{L^2}\sum_{i=1}^{I}\|\delta_i\|^2
													\leq \frac{1}{L^2C_P}\sum_{i=1}^{I}\|\nabla C_i\|^2.
										\end{aligned}
							\end{equation}
							By making use of the a priori bounds $\overline{C_i}\leq \sqrt{\overline{C_i^2}} \leq \sqrt{K}$ from Lemma \ref{bounded}, we can estimate the right hand side of \eqref{g19'} integrated over $S^{\perp}$ as follows
\begin{equation}\label{e20}
\begin{aligned}
\sum_{r=1}^{R}\left[\overline{\mathbf C}^{\boldsymbol{\alpha}^r} - \overline{\mathbf C}^{\boldsymbol{\beta}^r}\right]^2|S^{\perp}| 
&\leq C(\sqrt{K})\left|S^{\perp}\right| \leq \frac{C(K, \aa^r, \bb^r)}{L^2C_P}\sum_{i=1}^{I}\|\nabla C_i\|^2 \qquad (\text{use } \eqref{e19})\\
													&\leq \frac{1}{2}\sum_{i=1}^{I}\|\nabla C_i\|^2,
										\end{aligned}
							\end{equation}
if we choose $L$ to be big enough, e.g. $L^2 \geq \frac{2C(K, \aa^r, \bb^r)}{C_P}$. By combining \eqref{e17} and \eqref{e20} we obtain \eqref{g19'} with $\kappa = \frac{1}{2}\min\{1, \gamma\}$.
This ends the proof of Lemma \ref{average}.			\end{proof}

\end{description}
\medskip

\noindent\textit{Continuation of the proof of Lemma \ref{lem:intermediate}}\hfill\\
 Now, by combining the estimates \eqref{g16}, \eqref{g17} and \eqref{g19}, and the assumption \eqref{intermediate} we obtain
\begin{equation}\label{step3}
	\frac 12 \mathcal{D}(\cc) \geq \frac{K_1K_3H_6}{2K_2}\left(\mathcal{E}(\overline{\cc}) - \mathcal{E}(\ww)\right).
\end{equation}
Therefore, we finally obtain the desired entropy entropy-dissipation estimate $\mathcal{D}(\cc) \geq \lambda(\mathcal{E}(\cc) - \mathcal{E}(\ww))$ from \eqref{add2}, \eqref{step1} and \eqref{step3} with
\begin{equation*}
	\lambda = \frac 12 \min\left\{C_{LSI}d_{min}; \frac{K_1K_3H_6}{K_2}\right\}
\end{equation*}
where $K_1$ and $K_2$ are defined in \eqref{K1} and \eqref{K2} respectively, and $K_3$ is in Lemma \ref{average}.
\end{proof}
We are now ready to prove Theorem \ref{the:main}.
\begin{proof}[Proof of Theorem \ref{the:main}]\hfill\\
Thanks to Lemma \ref{lem:intermediate}, it remains to prove \eqref{intermediate} under the assumptions \eqref{ODEEqui} and \eqref{FarEqui}.
%
%
In order to do that, we exploit the ansatz 
\begin{equation}\label{g21}
\overline{C_i^2} = C_{i,\infty}^2(1+\mu_i)^2 \quad \text{with}\quad \mu_i \in [-1,\infty)\quad\text{ for all } \quad i=1,\ldots, I,
\end{equation}
or equivalently
\begin{equation*}
\overline{\mathbf C^2} = \mathbf C_{\infty}^2(\mathbf{1}+\boldsymbol{\mu} )^2
\end{equation*}
with $\mathbf{1} = (1,1,\ldots, 1)^{\top}\in \mathbb{R}^I$, $\boldsymbol{\mu}  = (\mu_1,\ldots, \mu_I)^{\top}$, and we use the notation $\mathbf C_{\infty}^2(\mathbf{1}+\boldsymbol{\mu} )^2 = (C_{i,\infty}^2(\mu_i^2 + 2\mu_i))_{i=1,\ldots, I}^{\top}$.
Note that $\overline{C_i^2} = \overline{c_i} \leq K$ (see Lemma \ref{bounded}), thus
\begin{equation*}
	-1\leq \mu_i \leq \mu_{max}< +\infty \quad \text{ for all } i=1,\ldots, I.
\end{equation*}
 By recalling that $\mathbf C^2 = \cc$ and $\mathbf C_{\infty}^2 = \cc_{\infty}$ and $\mathbb{Q}\,\overline{\cc} = \mathbb{Q}\,\cc_{\infty} = \M$, we have moreover the following algebraic constraints between the $\mu_1, \ldots, \mu_I$:
\begin{equation*}
\mathbb{Q}\,\mathbf C_{\infty}^2(\mathbf{1}+\boldsymbol{\mu} )^2 = \mathbb{Q}\,\mathbf C_{\infty}^2
\end{equation*}
or equivalently
\begin{equation*}\label{g21_1}
\mathbb{Q}\,\mathbf C_{\infty}^2(\boldsymbol{\mu}^2 + 2\boldsymbol{\mu}) = 0.
\end{equation*}
By denoting $\delta_i(x) = C_i(x) - \overline{C}_i$ the deviation of $C_i(x)$ to its average and by using \eqref{g21}, it follows from $\|\delta_i\|^2 = \overline{C_i^2} - \overline{C_i}^2$ that
									\begin{equation}\label{g22}
												\overline{C_i} = \sqrt{\overline{C_i^2}} - \frac{\|\delta_i\|^2}{\sqrt{\overline{C_i^2}} + \overline{C_i}} =: C_{i,\infty}(1+\mu_i) - \|\delta_i\|^2R(C_i),
									\end{equation}
where we denote $R(C_i) := \Bigl(\sqrt{\overline{C_i^2}} + \overline{C_i}\Bigr)^{-1}$ for all $i=1,\ldots, I$. We observe that $R(C_i)$ becomes unbounded when $\overline{C_i^2}\ge \overline{C_i}^2$ approaches zero. The possible occurrence of  such degenerate states (with arbitrarily little mass $\overline{C_i^2}$ 
of the species $\mathcal{A}_i$ being present)  
prevents a uniform use of the ansatz \eqref{g22}. 
Therefore, we have to distinguish two cases where $\overline{C_i^2}$ is either "big" or "small". Let $\varepsilon$ be as in assumption ii) of Theorem \ref{the:main}. We consider two cases.

\begin{description}
\item[{\bf Case 1}] When  $\overline{C_i^2} \geq \varepsilon^2$ for all $i=1,\ldots, I$. 												
In this case, we have
\begin{equation*}
R(C_i) = \left(\sqrt{\overline{C_i^2}} + \overline{C_i}\right)^{-1} \leq \frac{1}{\varepsilon}, \qquad \forall i=1,2,\ldots, I.
\end{equation*}
Thus, we can estimate the left hand side of \eqref{intermediate} for a $\theta\in (0,1)$ to be chosen (and by using Poincar\'e's inequality, the boundedness of $\|\delta_i\|$ and the boundedness of $\mu_i$), as follows
\begin{align}\label{g23}
\qquad\qquad\ \text{LSH of (\ref{intermediate})}&\geq C_P\sum_{i=1}^{I}\|\delta_i\|^2 + \theta\sum_{r=1}^{R}\Biggl[\prod\limits_{i=1}^{I}\left(C_{i,\infty}(1+\mu_i) - \|\delta_i\|^2R(C_i)\right)^{\alpha_i^r}\nonumber\\
&\hspace{3.7cm}- \prod\limits_{i=1}^{I}\left(C_{i,\infty}(1+\mu_i) - \|\delta_i\|^2R(C_i)\right)^{\beta_i^r}\Biggr]^2\nonumber\\
&\geq C_P\sum_{i=1}^{I}\|\delta_i\|^2 + \frac 12\theta\sum_{r=1}^{R}\left[\mathbf C_{\infty}^{\boldsymbol{\alpha}^r}(\mathbf{1} + \boldsymbol{\mu})^{\boldsymbol{\alpha}^r} -  \mathbf C_{\infty}^{\boldsymbol{\beta}^r}(\mathbf{1} + \boldsymbol{\mu})^{\boldsymbol{\beta}^r}\right]^2\nonumber - \theta\, C(\varepsilon, K)\sum_{i=1}^{I}\|\delta_i\|^2\nonumber\\
& \geq \frac 12\theta\min_{r=1,\ldots, R}\left\{\mathbf C_{\infty}^{\boldsymbol{\alpha}^r}\right\}^2\sum_{r=1}^{R}\left[(\mathbf{1} + \boldsymbol{\mu})^{\boldsymbol{\alpha}^r} -  (\mathbf{1}
+ \boldsymbol{\mu})^{\boldsymbol{\beta}^r}\right]^2,
\end{align}		
where we have used the detailed balance condition, $\CC_{\infty}^{\aa^r} = \CC_{\infty}^{\bb^r}$ and chosen $\theta \in (0,1)$ such that $\theta C(\varepsilon, K) \leq C_P$ where $C(\varepsilon, K)$ is a constant explicitly depends on $\varepsilon$ and $K$, \textcolor{black}{and also on $\aa^r$ and $\bb^r$}. Now, by using \eqref{ODEEqui} rewritten 
in the variables \eqref{g21}, i.e. $\sqrt{\overline{c_i}}=\sqrt{c_{i,\infty}}(1+\mu_i)$ and 
\begin{equation}\label{ODEEquimu}
\sum_{r=1}^{R}\left[(1+\boldsymbol{\mu})^{\aa^r} - (1+\boldsymbol{\mu})^{\bb^r}\right]^2 \geq H_4\sum_{i=1}^{I}\mu_i^2
\end{equation}
for $\boldsymbol{\mu}\in [-1,\mu_{max}]^{I}$, we continue to estimate \eqref{g23} below
\begin{equation*}
	\begin{aligned}
		\text{RHS of \eqref{g23}} &\geq \theta\min_{r=1,\ldots, R}\left\{\mathbf C_{\infty}^{\boldsymbol{\alpha}^r}\right\}^2H_4\sum_{i=1}^{I}\mu_i^2\\
		&= \theta\min_{r=1,\ldots, R}\left\{\mathbf C_{\infty}^{\boldsymbol{\alpha}^r}\right\}^2H_4\sum_{i=1}^{I}\frac{\left(\sqrt{\overline{C_i^2}} - C_{i,\infty}\right)^2}{C_{i,\infty}^2} \qquad (\text{by using } \eqref{g21})\\
		&\geq \left(\theta\min_{r=1,\ldots, R}\left\{\mathbf C_{\infty}^{\boldsymbol{\alpha}^r}\right\}^2\frac{H_4}{\max_{i}C_{i,\infty}^2}\right) \sum_{i=1}^{I}\left(\sqrt{\overline{C_i^2}} - C_{i,\infty}\right)^2.
	\end{aligned}
\end{equation*}
Therefore, \eqref{intermediate} follows with
\begin{equation*}
	H_6 = \theta\min_{r=1,\ldots, R}\left\{\mathbf C_{\infty}^{\boldsymbol{\alpha}^r}\right\}^2\frac{H_4}{\max_{i}C_{i,\infty}^2}.
\end{equation*}

%
%
\vskip .5cm
\item[{\bf Case 2}] When $\overline{C_{i_0}^2} \leq \varepsilon^2$ for some ${i_0}\in \{1,\ldots, I\}$. In this case, we only need to bound the right hand side of \eqref{intermediate} above since the left hand side of \eqref{intermediate} is bounded below by a positive constant in terms of the assumed inequality \eqref{FarEqui}
and, thus, \eqref{intermediate} holds true for a sufficiently small constant $H_6$.

Thanks to the boundedness of averaged concentrations $\overline{c_i} \leq K$ in Lemma \ref{bounded} and consequently  $c_{i,\infty}\leq K$, we have
\begin{equation}\label{2star}
\sum_{i=1}^{I}\left(\sqrt{\overline{C_i^2}} - C_{i,\infty}\right)^2 \leq 2\sum_{i=1}^{I}\left(\overline{C_i^2} + C_{i,\infty}^2\right) \leq 4IK.
%
\end{equation}
Hence, it follows from \eqref{FarEqui} that
\begin{equation*}
	\text{LHS of (\ref{intermediate})} \geq H_5 \geq \frac{H_5}{4IK}\sum_{i=1}^{I}\left(\sqrt{\overline{C_i^2}} - C_{i,\infty}\right)^2
\end{equation*}
or equivalently \eqref{intermediate} holds with
\begin{equation*}
	H_6 = \frac{H_5}{4IK}.
\end{equation*}
\end{description}
This finishes the proof the Theorem \ref{the:main}.
\end{proof}

\begin{remark}[How to prove \eqref{ODEEqui} and \eqref{FarEqui}]\hfill\\
The reason that prevents us to prove \eqref{ODEEqui} and \eqref{FarEqui} for general systems of the form \eqref{SS}--\eqref{kr} is that we need some explicit structure of the conservation laws. However, to categorise the structure of the conservation laws of general systems is unclear. We will show for the two specific systems in Sections \ref{sec:3} and \ref{sec:4} that once the conservation laws are explicitly known, we can establish the inequalities \eqref{ODEEqui} and \eqref{FarEqui} explicitly and thus complete the explicit proof of the entropy entropy-dissipation inequality \eqref{MainEstimate}. 

This remark provides an outline of how to obtain \eqref{ODEEqui} and \eqref{FarEqui} for a specific system (see Sections \ref{sec:3} and \ref{sec:4} for details).

\begin{itemize}
	\item[i)] From the conservation law constraints assumed in \eqref{ODEEqui} and the ansatz \eqref{g21}, we have
\begin{equation*}
\mathbb Q \mathbf C_{\infty}^2(\mm^2 + 2\mm) = 0
\end{equation*}
\textcolor{black}{in which $\mathbf C_{\infty}^2(\mm^2 + 2\mm) = (C_{i,\infty}^2(\mu_i^2 + 2\mu_i))_{i=1,\ldots, I}^{\top}$}, and consequently some sign relations between $\mu_i$, which allow to prove \eqref{ODEEqui} (see Lemmas \ref{InterInequal} and \ref{Chain_InterInequal}).	 	
		 
		
	\item[ii)] To show \eqref{FarEqui} for states, where (at least) one of the concentration contains less than a sufficiently small amount of mass, let's say $\overline{C_{i_0}^2} \leq \varepsilon^2$, we distinguish two cases depending on  diffusion being "big" or "small". In the former case, the lower bound $H_6$ follows from the diffusion part $\sum\limits_{i=1}^{I}\|\nabla C_i\|^2$, while in the latter case the reaction part $\sum\limits_{r=1}^{R}\left(\overline{\mathbf C}^{\aa^r} - \overline{\mathbf C}^{\bb^r}\right)^2$ gives the lower bound $H_6$ (see Lemmas \ref{lem:lowerbound_single} and \ref{lem:lowerbound_chain}).
\end{itemize}
We believe that the strategies i) and ii) are applicable to a large class of systems once the conservation laws are explicitly given. We refer the reader to e.g. \cite{DFT16, BFL, JanThesis, FL} where (parts of) the above strategies were already apply.

\end{remark}

\section{A single reversible reaction - Proof of Theorem \ref{mainsingle}}\label{sec:3}
In this section, we will prove Theorem \ref{mainsingle} and in particular the entropy entropy-dissipation estimate \eqref{EEDsingle} with an explicit constant $\lambda_1$, which yields convergence to equilibrium for a single reversible reaction of the form
\begin{equation*}
\alpha_1\mathcal{A}_1 + \alpha_2\mathcal{A}_2 + \ldots + \alpha_I\mathcal{A}_I \leftrightharpoons \beta_1\mathcal{B}_1 + \beta_2\mathcal{B}_2 + \ldots + \beta_J\mathcal{B}_J
\end{equation*}
with stoichiometric coefficients $\alpha_i, \beta_j \geq 1$ for $i=1,\ldots, I$ and $j=1,\ldots, J$ for any $I, J\geq 1$. For the sake of convenience and w.l.o.g. we rescale the forward and backward reaction rate constants to $k_{f} = k_b = 1$. An explicit entropy entropy-dissipation estimate was left open in \cite{MHM} whenever $I+J\geq 3$.  
			
			The reaction is assumed to take place in reaction vessel, i.e. in a bounded domain $\Omega\subset \mathbb{R}^n, n\geq 1$ with sufficiently smooth boundary $\partial\Omega$ (e.g. $\partial\Omega\in C^{2+\epsilon}$ for some $\epsilon >0$). The mass action reaction-diffusion system reads as
\begin{equation}\label{single}
		\begin{cases}
		\partial_ta_i - \mathrm{div}(d_{a,i}\nabla a_{i}) = -\alpha_i(\boldsymbol{a}^{\boldsymbol{\alpha}} - \boldsymbol{b}^{\boldsymbol{\beta}}), &\quad \text{ in } \Omega\times\mathbb{R}_+,\quad i=1,2,\ldots, I, \\
	\partial_tb_j- \mathrm{div}(d_{b,j}\nabla b_{j}) = \beta_j(\boldsymbol{a}^{\boldsymbol{\alpha}} - \boldsymbol{b}^{\boldsymbol{\beta}}), & \quad \text{ in } \Omega\times\mathbb{R}_+,\quad j=1,2,\ldots, J,\\
	\nabla a_i\cdot \nu = \nabla b_j\cdot \nu = 0, &\quad \text{ on } \partial\Omega\times\mathbb{R}_+,\quad i=1,\ldots, I,\; j=1,\ldots, J,\\
	a(x,0) = a_0(x), \quad b(x,0) = b_0(x), &\quad \text{ in }\Omega,
	\end{cases}
\end{equation}
where $ d_{a,i},  d_{b,j}$ are diffusion coefficients satisfying
\begin{equation*}\label{bound_diff}
	0< d_{min} \leq  d_{a,i}, \,  d_{b,j} \leq d_{max} < +\infty \qquad \forall  i=1,\ldots, I,\; j=1,\ldots, J,
\end{equation*}
$\boldsymbol{a} = (a_1, a_2, \ldots, a_I)$, $\boldsymbol{b} = (b_1, b_2, \ldots, b_J)$ are the vectors of left- and right-hand-side concentrations, $\boldsymbol{\alpha} = (\alpha_1, \alpha_2, \ldots, \alpha_I) \in [1,+\infty)^{I}$ and $\boldsymbol{\beta} = (\beta_1, \beta_2, \ldots, \beta_J) \in [1,+\infty)^{J}$ are the corresponding vectors of stoichiometric coefficients and we recall the notation
\begin{equation*}
\boldsymbol{a}^{\boldsymbol{\alpha}} = \prod\limits_{i=1}^{I}a_i^{\alpha_i} \quad \text{ and } \quad \boldsymbol{b}^{\boldsymbol{\beta}} = \prod\limits_{j=1}^{J}b_j^{\beta_j}.
\end{equation*}
			
The aim of this section is to apply the method proposed in Section \ref{sec:2} to prove the entropy entropy-dissipation estimate \eqref{EEDsingle}, which implies explicit convergence to equilibrium for the system \eqref{single}. First, we shall specify the mass conservation laws for \eqref{single}, which are essential to our strategy. Then, \eqref{single} is shown to satisfy the assumptions $(\mathbf{A1})$ and $(\mathbf{A2})$, that is \eqref{single} satisfies the detailed balance condition and has no boundary equilibrium. Finally, we prove  the main result of this section Theorem \ref{mainsingle}.

\begin{lemma}[Mass conservation laws]\label{masssingle}\hfill\\
The system \eqref{single} obeys $I+J-1$ linear independent mass conservation laws.
With respect to the general notation of conservation laws \eqref{g9}, the matrix $\mathbb{Q}$ can be chosen as
\begin{equation*}
\mathbb{Q} = [v_1, \ldots, v_J, w_2, \ldots, w_I]^{\top}\in \mathbb{R}^{(I+J-1)\times (I+J)}
\end{equation*}
where the zero left-eigenvectors $v_j$ and $w_i$ are defined by
\begin{equation}\label{vj}
v_j = \biggl(\underbrace{\frac{1}{\alpha_1}, 0, \ldots, 0, \frac{1}{\beta_j}}_{I+j}, 0, \ldots, 0\biggr) \qquad \text{ and } \qquad  w_i = \biggl(\underbrace{\underbrace{0,\ldots, 0, \frac{1}{\alpha_i}}_{i}, 0, \ldots, 0, \frac{1}{\beta_1}}_{I+1}, 0, \ldots, 0\biggr),
\end{equation}
for $1\leq j \leq J$ and $2\leq i\leq I$. 
\end{lemma}
\begin{proof}
From \eqref{single}, by dividing the equation of $a_i$ by $\alpha_i$ and the equation of $b_j$ by $\beta_j$, summation and integration over $\Omega$ yields, thanks to the homogeneous Neumann boundary condition,
\begin{equation*}
\frac{d}{dt}\int_{\Omega}\left(\frac{a_i(x,t)}{\alpha_i} + \frac{b_j(x,t)}{\beta_j}\right)dx = 0 \qquad \forall t>0.
\end{equation*}
Hence, after introducing the nonnegative partial masses $M_{i,j}:= \int_{\Omega}\bigl(\frac{a_{i,0}(x)}{\alpha_i} + \frac{b_{j,0}(x)}{\beta_j}\bigr)dx$, we observe that system \eqref{single} obeys the following $IJ$ mass conservation laws
\begin{equation}\label{ConsMij}
\frac{\overline{a_i}(t)}{\alpha_i} + \frac{\overline{b_j}(t)}{\beta_j} = M_{i,j}, \qquad \forall t>0,\quad  \forall i=1,\ldots, I,\quad  \forall j=1,\ldots, J,
\end{equation}
where we recall the notation $\overline{a_i} = \int_{\Omega}a_i(x)dx$. The following $I+J-1$ conservations laws 
\begin{equation*}
\frac{\overline{a_1}}{\alpha_1} + \frac{\overline{b_j}}{\beta_j} = M_{1,j}\quad \text{ and } \quad \frac{\overline{a_i}}{\alpha_i} + \frac{\overline{b_1}}{\beta_1} = M_{i,1} \qquad 
\text{for} \quad j=1,\ldots, J, \quad i=2,\ldots, I.
\end{equation*}
are linear independent since the corresponding zero left-eigenvectors $v_j$ and $w_i$ as given in \eqref{vj} are linear independent by construction. Moreover, all other conservation laws can be expressed in terms of these $I+J-1$ conservation laws due to
\begin{equation*}
\frac{\overline{a_i}}{\alpha_i} + \frac{\overline{b_j}}{\beta_j} = \left(M_{i,1} - \frac{\overline{b_1}}{\beta_1}\right) + \left(M_{1,j} - \frac{\overline{a_1}}{\alpha_1}\right) = M_{i,1} + M_{1,j} - M_{1,1}.
\end{equation*}
\end{proof}
			
\begin{remark}\label{Lem:FixInitialMass}
It follows from the Lemma \ref{masssingle} that {the $I+J-1$ coordinates of the initial mass vector $\M$ are determined by prescribing $M_{1,j}$ for $1\leq j\leq J$ and $M_{i,1}$ for $i=2,\ldots, I$}. Therefore, by calling the initial mass vector $\M$ fixed, we mean that those coordinates are prescribed.
\end{remark}

\begin{remark}	
Another useful set of conservation laws follows from dividing the equation for $a_i$ by $\alpha_i$ and the equation for $a_k$ by $\alpha_k$ for $1\le i\neq k \le I$:
\begin{equation*}\label{ConsNij}
\int_{\Omega}\left(\frac{a_{i}(t,x)}{\alpha_i} -  \frac{a_{k}(t,x)}{\alpha_k }\right)dx = N_{i,k}:= \int_{\Omega}\left(\frac{a_{i,0}(x)}{\alpha_i} -  \frac{a_{k,0}(x)}{\alpha_k }\right)dx, 
\qquad \forall t\geq 0,\quad 
1\le i\neq k \le I,
\end{equation*}
It's also useful to observe that
\begin{equation}\label{3.13'}
N_{i,k} = M_{i,j} - M_{k,j}, \qquad \forall 1\leq j \leq J.
\end{equation}
\end{remark}

\begin{lemma}[Unique constant positive detailed balance equilibrium]\label{Equilibrium}\hfill\\
For any fixed positive initial mass vector $\M$, the system \eqref{single} possesses a unique positive detailed balance equilibrium $(\boldsymbol{a}_{\infty}, \boldsymbol{b}_{\infty})\in (0,+\infty)^{I+J}$ solving
\begin{equation}\label{Equilibrium1} 
\begin{cases}
\dfrac{a_{i,\infty}}{\alpha_i} + \dfrac{b_{j,\infty}}{\beta_j} = M_{i,j}, \qquad i=1,\ldots, I,\quad j=1,\ldots, J,\\[2mm]
\boldsymbol{a}_{\infty}^{\boldsymbol{\alpha}} = \boldsymbol{b}_{\infty}^{\boldsymbol{\beta}}.
\end{cases}
\end{equation}
Consequently, system \eqref{single} satisfies the assumptions $(\mathbf{A1})$ and $(\mathbf{A2})$.
\end{lemma}
\begin{proof}
\textcolor{black}{
Without loss of generality, we may assume that
\[
M_{1,1} = \min_{i=1,\ldots, I}\{M_{i,1}\},
\]
which implies in \eqref{3.13'} that $N_{i,1} = M_{i,1} - M_{1,1} \geq 0$ for $i=2,\ldots, I$. 
In the following, we denote by $(\boldsymbol{a}_{\infty}, \boldsymbol{b}_{\infty})$ a possible solution
to \eqref{Equilibrium1}.}
From \eqref{Equilibrium1} and \eqref{3.13'} it follows  for all $i>1$ that
\begin{equation}
\frac{a_{i,\infty}}{\alpha_i} - \frac{a_{1,\infty}}{\alpha_1} = M_{i,k} - M_{1,k} = N_{i,1} \ge0
\end{equation}
Thus, for any possible solution, we observe that   
\begin{equation*}\label{a1mono}
\boldsymbol{a}_{\infty}^{\boldsymbol{\alpha}}= \prod\limits_{i=1}^{I}a_{i,\infty}^{\alpha_i} = 
a_{1,\infty}^{\alpha_1} \prod\limits_{i=2}^{I} \left(\alpha_i N_{i,1} + \frac{\alpha_i}{\alpha_1} a_{1,\infty}\right)^{\alpha_i} =: f(a_{1,\infty}),
\end{equation*}
\textcolor{black}{
is a positive, strictly monotone increasing function $f(a_{1,\infty})>0$
in $a_{1,\infty}\in(0,+\infty)$ with $f(0)=0$ and $\lim_{z\to\infty}f(z) = +\infty$. 
Similarly, we deduce from \eqref{Equilibrium1} that $b_{j,\infty}(a_{1,\infty}) = \beta_j M_{1,j} - \frac{\beta_j}{\alpha_1} a_{1,\infty}$ and $b_{j,\infty}(0)>0$ since $\M>0$ by assumption. Therefore,
}
\begin{equation*}\label{b1mono}
\boldsymbol{b}_{\infty}^{\boldsymbol{\beta}}=\prod\limits_{j=1}^{J}b_{j,\infty}^{\beta_j} = 
\prod\limits_{j=1}^{J} \left(\beta_j M_{1,j}- \frac{\beta_j}{\alpha_1} a_{1,\infty}\right)^{\beta_j} =: g(a_{1,\infty}),
\end{equation*}
\textcolor{black}{
is a strictly monotone decreasing function in $a_{1,\infty}$ 
with $g(0) >0$. By defining $0< \hat a_{1,\infty}:=\min_{j=1,\ldots, J}\left\{\alpha_1 M_{1,j}\right\}$,
we note that $b_{j^*,\infty}(\hat a_{1,\infty})=0$ for some $j^*\in\{1,\ldots, J\}$ and $b_{j,\infty}(a_{1,\infty})>0$ for all $a_{1,\infty}\in (0, \hat a_{1,\infty})$ and $j\in\{1,\ldots, J\}$.
Thus $g(\hat a_{1,\infty}) = 0$ while 
$g(a_{1,\infty})>0$ for $a_{1,\infty}\in (0, \hat a_{1,\infty})$
and the equation $f(a_{1,\infty}) = g(a_{1,\infty})$
has a unique positive solution $a_{1,\infty}\in (0, \hat a_{1,\infty})$. By setting 
\begin{equation*}
a_{i,\infty} = \alpha_{i}N_{i,1} + \frac{\alpha_i}{\alpha_1}a_{1,\infty} \quad \text{ and } \quad b_{j,\infty} = \beta_jM_{1,j} - \frac{\beta_j}{\alpha_1}a_{1,\infty}
\end{equation*}
for $i=2,\ldots, I$ and $j=1,\ldots, J$ and by recalling $N_{i,1} \geq 0$, this yields 
a unique positive equilibrium $(\boldsymbol{a}_{\infty}, 
\boldsymbol{b}_{\infty})\in (0,+\infty)^{I+J}$, 
which satisfies \eqref{Equilibrium1} by construction.
}
			
It is straightforward that the assumption $(\mathbf{A1})$ holds. The assumption $(\mathbf{A2})$ can be easily verified by assuming $a_{i_0,\infty} = 0$ for some $1\leq i_0 \leq I$, which implies $\boldsymbol{a}_{\infty}^{\boldsymbol{\alpha}} = 0$. Then, by \eqref{Equilibrium1}, $b_{j,\infty} = \beta_j M_{i_0,j}>0$ for all $j=1,2,\ldots, J$, and thus $\boldsymbol{b}_{\infty}^{\boldsymbol{\beta}} > 0$, which is a contradiction. 
We conclude $a_{i,\infty} > 0$ for all $i=1,2,\ldots, I$ and similarly, $b_{j,\infty} >0$ for all $j=1,2,\ldots, J$. Therefore, the system \eqref{single} has no boundary equilibrium.

\end{proof}

The entropy functional for system \eqref{single} writes as
\begin{equation*}
			\mathcal{E}(\boldsymbol{a}, \boldsymbol{b}) = \sum_{i=1}^{I}\int_{\Omega}(a_i\log{a_i} - a_i + 1)dx + \sum_{j=1}^{J}\int_{\Omega}(b_j\log{b_j} - b_j + 1)dx
\end{equation*}
and the entropy-dissipation writes as
\begin{equation*}
			\mathcal{D}(\boldsymbol{a}, \boldsymbol{b}) = \sum_{i=1}^{I}\int_{\Omega} d_{a,i}\frac{|\nabla a_i|^2}{a_i}dx + \sum_{j=1}^{J}\int_{\Omega} d_{b,j}\frac{|\nabla b_j|^2}{b_j}dx + \int_{\Omega}(\boldsymbol{a}^{\boldsymbol{\alpha}} - \boldsymbol{b}^{\boldsymbol{\beta}})\log{\frac{\boldsymbol{a}^{\boldsymbol{\alpha}}}{\boldsymbol{b}^{\boldsymbol{\beta}}}}dx.
\end{equation*}

\begin{lemma}\label{InterInequal}
	Let $\mu_i\in [-1,+\infty)$, $i=1,\ldots, I$ and $\xi_j\in [-1,+\infty)$, $j=1,\ldots, J$ satisfy
	\begin{equation}\label{i2}
		A_{i,\infty}^2\,\mu_i(\mu_i + 2) + B_{j,\infty}^2\,\xi_j(\xi_j + 2) = 0, 
		\qquad \text{for all} \quad 1\leq i\leq I, \quad 1\leq j \leq J.
	\end{equation}
	Then,  the following inequality holds
	\begin{equation}\label{Inter1}
	\biggl(\prod_{i=1}^{I}(1+\mu_i)^{\alpha_i} - \prod_{j=1}^{J}(1+\xi_j)^{\beta_j}\biggr)^2 \geq H_4^{single} \biggl(\sum_{i=1}^{I}\mu_i^2 + \sum_{j=1}^{J}\xi_j^2\biggr)
	\end{equation}
	where
	\begin{equation}\label{H4_single}
		H_4^{single} = \frac{1}{\max\{I, J\}}.
	\end{equation}
\end{lemma}
\begin{remark}
	Note that $\mu_i$ and $\xi_j$ are derived from perturbation argument like \eqref{g21}, i.e. 
	\begin{equation*}
		\overline{A_i^2} = A_{i,\infty}^2(1+\mu_i)^2 \quad \text{ and } \quad \overline{B_j^2} = B_{j,\infty}^2(1+\xi_j)^2.
	\end{equation*}
\end{remark}
\begin{proof}
	The proof of \eqref{Inter1} relies only on sign relations between $\mu_i$ and $\xi_j$ entailed by the conservation laws \eqref{i2}.
	Since $\mu_i, \xi_j \in [-1,+\infty)$, it follows from \eqref{i2} that $\mu_i$ and $\xi_j$ must have opposite signs. In particular, if $\mu_i \geq 0$ for some $i=1,2,\ldots, I$, then $\xi_j \leq 0$ for all $j=1,2,\ldots, J$ and in return $\mu_i \geq 0$ for all $i=1,2,\ldots, I$. Therefore, we only have to consider two cases:
	\begin{itemize}
		\item[(I)] $\mu_i \geq 0$ and $\xi_j \leq 0$ for all  $i=1,2,\ldots, I$ and $j=1,2,\ldots, J$, or
		\item[(II)] $\mu_i \leq 0$ and $\xi_j \geq 0$ for all  $i=1,2,\ldots, I$ and $j=1,2,\ldots, J$.
	\end{itemize}
	We will only prove case (I) since case (II) can be treated analogue. Now since $\mu_i\geq 0$ and $-1\leq \xi_j\leq 0$, we can apply Jensen's  inequality for any $1\leq i_* \leq I$ and $1\leq j_* \leq J$ to estimate
	\begin{equation*}
	\begin{aligned}
	\biggl|\prod_{i=1}^{I}(1+\mu_i)^{\alpha_i} - \prod_{j=1}^{J}(1+\xi_j)^{\beta_j}\biggr| &\geq \prod_{i=1}^{I}(1+\mu_i)^{\alpha_i} - \prod_{j=1}^{J}(1+\xi_j)^{\beta_j}
	\geq \prod_{i=1}^{I}(1+\mu_i) - \prod_{j=1}^{J}(1+\xi_j)\\
	&\geq (1+\mu_{i_*}) - (1+\xi_{j_*})\\
	&= \mu_{i_*} - \xi_{j_*} \geq 0.
	\end{aligned}
	\end{equation*}
	Hence, for any indices $1\leq i_* \leq I$ and $1\leq j_* \leq J$
	\begin{equation*}
	\biggl(\prod_{i=1}^{I}(1+\mu_i)^{\alpha_i} - \prod_{j=1}^{J}(1+\xi_j)^{\beta_j}\biggr)^2 \geq (\mu_{i_*} - \xi_{j_*})^2 \geq \mu_{i_*}^2 + \xi_{j_*}^2
	\end{equation*}
	since $\mu_{i^*}\mu_{j^*} \leq 0$, and \eqref{Inter1} follows immediately.			
\end{proof}

\begin{lemma}\label{lem:lowerbound_single}
{Under the assumptions of Theorem \ref{mainsingle}}, there exists $0<\varepsilon\ll 1$ such that if $\overline{A_i^2} \leq \varepsilon^2$ for some $i=1,\ldots, I$ or $\overline{B_j^2} \leq \varepsilon^2$ for some $j=1,\ldots, J$, then 
	\begin{equation}\label{lowerbound_single}
		\sum_{i=1}^{I}\|\nabla A_i\|^2 + \sum_{j=1}^{J}\|\nabla B_j\|^2 + \left(\overline{\boldsymbol A}^{\aa} - \overline{\boldsymbol B}^{\bb}\right)^2 \geq H_5^{single}
	\end{equation}  
	where
	\begin{equation}\label{H_5single}
		H_5^{single} = \min\left\{\frac{C_p\varepsilon^2}{\max_i \alpha_i}; \frac{C_p\varepsilon^2}{\max_j\beta_j}; \frac 14 \min_{i}\prod_{j=1}^{J}\left[\frac{\beta_jM_{i,j}}{2}\right]^{\beta_j}; \frac 14 \min_{j}\prod_{i=1}^{I}\left[\frac{\alpha_iM_{i,j}}{2}\right]^{\alpha_i} \right\}.
	\end{equation}
\end{lemma}
\begin{remark}
	The value of $\varepsilon$ can be computed explicitly, see \eqref{e38}.
\end{remark}
\begin{proof}
	Without loss of generality, we can assume that $\overline{A_{i_0}^2} \leq \varepsilon^2$ for some $1\leq i_0 \leq I$. Define
	\begin{equation*}
		\delta_i(x) = A_i(x) - \overline{A}_i \qquad \text{ and } \qquad \eta_j(x) = B_j(x) - \overline{B}_j
	\end{equation*}
	for $i=1,\ldots, I$ and $j=1,\ldots, J$.
%
%
	We will to consider two subcases due to the amount of  diffusion represented by the values of $\|\delta_i\|^2$ and $\|\eta_j\|^2$:
	\begin{itemize}[topsep=5pt, leftmargin=5mm]
		\item[(i)] (Diffusion is dominant.) 
		Consider that $\|\delta_{i_*}\|^2 \geq \frac{\varepsilon^2}{\alpha_{i_0}}$ for some $1\leq i_*\leq I$ or $\|\eta_{j_*}\|^2 \geq \frac{\varepsilon^2}{\alpha_{i_0}}$ for some $1\leq {j}_{*}\leq J$.
		Then, thanks to the Poincar\'{e} inequality, the left hand side of \eqref{lowerbound_single} obviously bounded below by
		\begin{equation}\label{tt5}
		\sum_{i=1}^{I}\|\nabla A_i\|^2 + \sum_{j=1}^{J}\|\nabla B_j\|^2 + \left(\overline{\boldsymbol A}^{\aa} - \overline{\boldsymbol B}^{\bb}\right)^2
		\geq C_P\biggl(\sum_{i=1}^I\|\delta_i\|^2 + \sum_{j=1}^{J}\|\eta_j\|^2\biggr)\geq \frac{C_P\varepsilon^2}{\alpha_{i_0}} \geq \frac{C_p\varepsilon^2}{\max_{i}\alpha_i}.
		\end{equation}
		\item[(ii)] (Reactions are dominant.)
		Consider that $\|\delta_i\|^2 \leq \frac{\varepsilon^2}{\alpha_{i_0}}$ for all $i = 1,\ldots,I$ and $\|\eta_j\|^2 \leq \frac{\varepsilon^2}{\alpha_{i_0}}$ for all $j=1,\ldots,J$. Then, by using the mass conservation \eqref{ConsMij}, 							\begin{equation*}\label{e34}
		\frac{\overline{A_{i_0}^2}}{\alpha_{i_0}} + \frac{\overline{B_j^2}}{\beta_j} = M_{i_0,j},
		\end{equation*}
		we have with $\overline{A_{i_0}^2} \leq \varepsilon^2$
		\begin{equation*}\label{e35}
		\overline{B_j^2} = \beta_j\biggl(M_{i_0,j} - \frac{\overline{A_{i_0}^2}}{\alpha_{i_0}}\biggr) \geq \beta_j\left(M_{i_0,j} - \frac{\varepsilon^2}{\alpha_{i_0}}\right),
		\end{equation*}
		for all $j = 1, \ldots, J$. Hence, for all $j = 1,\ldots, J$,
		\begin{equation*}\label{e36}
		\overline{B_j}^2 = \overline{B_j^2} - \|\eta_j\|^2 \geq \beta_jM_{i_0,j} - \frac{\beta_j+1}{\alpha_{i_0}}\varepsilon^2.
		\end{equation*}
		Then, we estimate the left hand side of \eqref{lowerbound_single} with an elementary inequality $(a-b)^2 \geq \frac{1}{2}a^2 - b^2$, $\overline{A_{i_0}^2} \leq \varepsilon^2$ and $\overline{A_i}^{2\alpha_i}\le \overline{A_i^2}^{\alpha_i} \le (\alpha_i M_{i,1})^{\alpha_i}$ for all $i\neq i_0$ as follows:
		\begin{equation}\label{e37_0}
		\begin{aligned}
		\biggl(\,\prod\limits_{i=1}^{I}\overline{A_i}^{\alpha_i} - \prod\limits_{j=1}^{J}\overline{B_j}^{\beta_j}\biggr)^2
		&\geq \frac{1}{2}\prod\limits_{j=1}^{J}\overline{B_j}^{2\beta_j} - \prod\limits_{i=1}^{I}\overline{A_i}^{2\alpha_i}\\
		&\geq \frac{1}{2}\prod\limits_{j=1}^{J}\left[\beta_jM_{i_0,j} - \frac{\beta_j+1}{\alpha_{i_0}}\varepsilon^2\right]^{\beta_j} - \varepsilon^2\prod\limits_{i=1,i\not=i_0}^{I}(\alpha_i M_{i,1})^{\alpha_i}.
		\end{aligned}
		\end{equation}
		We now choose $\varepsilon$ small enough such that both inequalities
		\begin{equation*}
		\beta_jM_{i_0,j} - \frac{\beta_j+1}{\alpha_{i_0}}\varepsilon^2 \geq \frac{\beta_j M_{i_0,j}}{2}  \quad \text{ for all } \quad j = 1,\ldots, J
		\end{equation*}
		and
		\begin{equation}\label{e37_1}
		\frac{1}{2}\prod\limits_{j=1}^{J}\left[\frac{\beta_jM_{i_0,j}}{2}\right]^{\beta_j} - \varepsilon^2\prod\limits_{i=1,i\not=i_0}^{I}(\alpha_i M_{i,1})^{\alpha_i} \geq \frac{1}{4}\prod\limits_{j=1}^{J}\left[\frac{\beta_jM_{i_0,j}}{2}\right]^{\beta_j} \geq \frac{1}{4}\min_{i}\prod\limits_{j=1}^{J}\left[\frac{\beta_jM_{i,j}}{2}\right]^{\beta_j}
		\end{equation}
		hold, e.g. $\varepsilon$ fulfills
		\begin{equation}\label{e38}
		\varepsilon^2 \leq \min\biggl\{\min\limits_{1\leq j\leq J}\left\{\frac{\alpha_{i_0}\beta_jM_{i_0,j}}{4(\beta_j +1)}\right\}; \frac 14\biggl(\prod\limits_{i=1,i\not=i_0}^{I} (\alpha_i M_{i,1})^{\alpha_i}\biggr)^{-1}\prod\limits_{j=1}^{J}\left[\frac{\beta_jM_{i_0,j}}{2}\right]^{\beta_j}\biggr\}.
		\end{equation}
		It then follows from \eqref{e37_0} and \eqref{e37_1} that
		\begin{equation}\label{e37}
		\sum_{i=1}^{I}\|\nabla A_i\|^2 + \sum_{j=1}^{J}\|\nabla B_j\|^2 + \left(\overline{\boldsymbol{A}}^{\aa} - \overline{\boldsymbol{B}}^{\bb}\right)^2
		\geq \frac{1}{4}\min_{i}\prod\limits_{j=1}^{J}\left[\frac{\beta_jM_{i,j}}{2}\right]^{\beta_j}
		\end{equation}
		Combining \eqref{tt5} and \eqref{e37} we obtain the desired inequality \eqref{lowerbound_single}.
	\end{itemize}						
\end{proof}
\begin{proof}[Proof of Theorem \ref{mainsingle}]
It follows from Lemma \ref{Equilibrium} that the assumptions $(\mathbf{A1})$ and $(\mathbf{A2})$ hold for system \eqref{SysSingle}.	
	
From Theorem \ref{the:main}  for general systems and the two Lemmas \ref{InterInequal} and \ref{lem:lowerbound_single}, we obtain the desired entropy entropy-dissipation estimate \eqref{EEDsingle} for an explicit constant $\lambda_1>0$.			

The existence of a global (in time) renormalised solution $(\boldsymbol{a}, \boldsymbol{b})$, 
\textcolor{black}{and the fact that any renormalised solution satisfies the weak entropy entropy-dissipation inequality \eqref{weak_entropy}, follow from \cite{Fi14,Fis16}.
Then,} thanks to the inequality \eqref{EEDsingle} and a suitable Gronwall argument, see e.g. \cite{Wil} or more specifically \cite{FL}, we obtain the convergence to equilibrium in relative entropy
\begin{equation*}
	\mathcal{E}(\boldsymbol{a}(t), \boldsymbol{b}(t)) - \mathcal{E}(\boldsymbol{a}_\infty, \boldsymbol{b}_{\infty}) \leq e^{-\lambda_1t}(\mathcal{E}(\boldsymbol{a}_0, \boldsymbol{b}_0) - \mathcal{E}(\boldsymbol{a}_\infty, \boldsymbol{b}_{\infty})).
\end{equation*}
Finally, $L^1$-convergence to equilibrium of renormalised solutions follows from the Csisz\'{a}r-Kullback-Pinsker inequality in Lemma \ref{CKP}.
%
\end{proof}

\section{Reversible enzyme reactions - Proof of Theorem \ref{mainChain}}\label{sec:4}

In this section, we demonstrate the strategy in Section \ref{sec:2} for a chain of two reversible reactions modelling, for instance, reversible enzyme reactions. More precisely, we consider the system
\begin{equation}\label{enzyme}
	\mathcal{C}_1 + \mathcal{C}_2 \leftrightharpoons \mathcal{C}_3 \leftrightharpoons \mathcal{C}_4 + \mathcal{C}_5.
\end{equation}
In \cite{BiCoDe07} and \cite{BoPie10}, this reaction was studied in the context of performing a quasi-steady-state-approximation, i.e. that the releasing reaction rate constant from $\mathcal{C}_3$ to $\mathcal{C}_1+ \mathcal{C}_2$ and from $\mathcal{C}_3$ to $\mathcal{C}_4+\mathcal{C}_5$ are taken to diverge to infinitely. 
Here in contrast, for the sake of clear presentation we shall assume all the reaction constants to be one (this is without loss of generality since our proof works equally in the case of any positive reaction rate constants).

As in previous sections, we assume the reaction to occur on a bounded domain $\Omega\subset \mathbb{R}^n$ with smooth boundary $\partial\Omega$ and normalised volume $|\Omega| = 1$. By applying the law of mass action, the corresponding reaction-diffusion system of \eqref{enzyme} reads as
\begin{equation}\label{ChainReaction}
	\begin{cases}
		\partial_tc_1 - d_1\Delta c_1 = -c_1c_2 + c_3, &\text{ in } \Omega\times\mathbb{R}_+, \\
		\partial_tc_2 - d_2\Delta c_2 = -c_1c_2 + c_3, &\text{ in } \Omega\times\mathbb{R}_+, \\
		\partial_tc_3 - d_3\Delta c_3 = c_1c_2 + c_4c_5 - 2c_3, \quad&\text{ in } \Omega\times\mathbb{R}_+,\\
		\partial_tc_4 - d_4\Delta c_4 = -c_4c_5 + c_3, &\text{ in } \Omega\times\mathbb{R}_+,\\
		\partial_tc_5 - d_5\Delta c_5 = -c_4c_5 + c_3, &\text{ in } \Omega\times\mathbb{R}_+,\\
		\nabla c_i\cdot \nu =  0, \qquad\qquad i=1,2,\ldots, 5,\quad &\text{ on } \partial\Omega\times\mathbb{R}_+,\\
		c_i(x,0) = c_{i,0}(x), &\text{ in } \Omega,
	\end{cases}
\end{equation}
where $d_i>0$ for all $i=1,2,\ldots, 5$ are positive diffusion coefficients.
\medskip

This section is organised as follows: We first derive the mass conservation laws for \eqref{ChainReaction}, which play an essential role in our method. Then, we show that \eqref{ChainReaction} satisfies the assumptions $(\mathbf{A1})$ and $(\mathbf{A2})$. Finally, we apply the method presented in Section \ref{sec:2} to prove explicit convergence to equilibrium for \eqref{ChainReaction}. For the sake of convenience, we will denote by $\boldsymbol{c} = (c_1, c_2, c_3, c_4, c_5)$.
We begin by stating the mass conservation laws. The proof of the following Lemma \ref{masschain} is straightforward and thus omitted.
\begin{lemma}[Conservation laws for \eqref{ChainReaction}]\label{masschain}	\hfill\\
For $i\in \{1,2\}$ and $j\in \{4,5\}$, we have (after recalling the notation $\overline{c_i}(t) = \int_{\Omega}c_i(x,t)dx$)
\begin{equation*}\label{Chain_Mass}
\overline{c_i}(t) + \overline{c_j}(t) + \overline{c_3}(t) = \overline{c_{i,0}} + \overline{c_{j,0}} + \overline{c_{3,0}} =: M_{i,j},	\qquad \text{for all}\quad  t>0.
\end{equation*}
Among these four conservation laws, there are exactly three linear independent conservation laws.		
Moreover, in terms of the notations \eqref{SS} and \eqref{g9}, the matrix $\mathbb{Q}$ can be chosen as
		\begin{equation*}
					\mathbb{Q} = \begin{pmatrix}
								1&0&1&1&0\\
								1&0&1&0&1\\
								0&1&1&1&0\\
					\end{pmatrix}
					\in \mathbb{R}^{3\times 5}.
		\end{equation*}
\end{lemma}

\begin{remark}
We denote by $(M_{1,4}, M_{1,5}, M_{2,4}, M_{2,5})\in \mathbb{R}^4_+$ the {vector of conserved masses. Note that the positive initial mass vector $\M\in\mathbb{R}^3_+$ is given once three linear independent coordinates of $(M_{1,4}, M_{1,5}, M_{2,4}, M_{2,5})\in \mathbb{R}^4_+$ are prescribed.} 
%
\end{remark}

\begin{lemma}[Detailed balance equilibrium]\label{Chain_Equi}\hfill\\
			For any given positive initial mass vector $\M\in \mathbb{R}^3_+$, there exists a unique positive detailed balance equilibrium $\boldsymbol{c}_{\infty} = (c_{1,\infty}, c_{2,\infty},\ldots, c_{5,\infty})$ to \eqref{ChainReaction} satisfying
			\begin{equation*}\label{Chain_Equi_Eq}
						\begin{cases}
									c_{1,\infty}c_{2,\infty}  = c_{3,\infty},\\
									c_{4,\infty}c_{5,\infty} = c_{3,\infty},\\
									c_{i,\infty} + c_{j,\infty} + c_{3,\infty} = M_{i,j}, \qquad \forall i\in\{1,2\},\quad \forall j\in \{4,5\}.
						\end{cases}
\end{equation*}
Consequently, the system \eqref{ChainReaction} satisfies the assumptions $(\mathbf{A1})$ and $(\mathbf{A2})$.
\end{lemma}
\begin{proof}
The proof of this Lemma follows from similar calculations as the proof of Lemma \ref{Equilibrium} and is omitted.
\end{proof}

To prove the convergence to the equilibrium for \eqref{ChainReaction}, we consider the entropy functional
\begin{equation*}\label{Chain_Entropy}
	\mathcal{E}(\boldsymbol{c}) = \sum_{i=1}^{5}\int_{\Omega}(c_i\log{c_i} - c_i + 1)dx
\end{equation*}
and its entropy-dissipation
\begin{equation*}\label{Chain_EnDiss}
	\mathcal{D}(\boldsymbol{c}) = \sum_{i=1}^{5}\int_{\Omega}d_i\frac{|\nabla c_i|^2}{c_i}dx + \int_{\Omega}\left((c_1c_2 - c_3)\log{\frac{c_1c_2}{c_3}} + (c_4c_5 - c_3)\log{\frac{c_4c_5}{c_3}}\right)dx.
\end{equation*}

\begin{lemma} \label{Chain_InterInequal}
	Let $\mu_1,\ldots, \mu_5 \in [-1,+\infty)$ satisfy the following conservation laws
	\begin{equation}\label{mass_1}
		C_{1,\infty}^2\mu_1(\mu_1+2) + C_{3,\infty}^2\mu_3(\mu_3+2)  + C_{4,\infty}^2\mu_4(\mu_4+2) = 0,
	\end{equation}
	\begin{equation}\label{mass_2}
		C_{1,\infty}^2\mu_1(\mu_1+2) + C_{3,\infty}^2\mu_3(\mu_3+2)  + C_{5,\infty}^2\mu_5(\mu_5+2) = 0,
	\end{equation}
	and
	\begin{equation}\label{mass_3}
		C_{2,\infty}^2\mu_2(\mu_2+2) + C_{3,\infty}^2\mu_3(\mu_3+2)  + C_{4,\infty}^2\mu_4(\mu_4+2) = 0.
	\end{equation}
	These conservation laws correspond to those of Lemma \ref{masschain} (recalling the perturbation argument \eqref{g21}, i.e. $\overline{C_i^2} = C_{i,\infty}^2(1+\mu_i)^2$ for $i=1,\ldots, 5$).
	
	Then, there holds 
	\begin{equation}\label{ttt}
	\begin{gathered}
	\bigl((1+\mu_1)(1+\mu_2) -(1+\mu_3)\bigr)^2 + \bigl((1+\mu_4)(1+\mu_5) - (1+\mu_3)\bigr)^2
	\geq \frac{1}{12}\sum_{i=1}^{5}\mu_i^2.
	\end{gathered}
	\end{equation}
\end{lemma}
\begin{proof}
	This inequality is similar to \eqref{Inter1}. However, due to the different structure of mass conservation laws, we need to use a different proof.
	
	By the triangle inequality, the left hand side of \eqref{ttt} is bounded below by
	\begin{multline}\label{ttt_1}
	\bigl((1+\mu_1)(1+\mu_2) -(1+\mu_3)\bigr)^2 + \bigl((1+\mu_4)(1+\mu_5) - (1+\mu_3)\bigr)^2\\
	\geq \frac{1}{4}\biggl[\underbrace{\bigl((1+\mu_1)(1+\mu_2) -(1+\mu_3)\bigr)^2}_{=:G_1} + \underbrace{\bigl((1+\mu_4)(1+\mu_5) - (1+\mu_3)\bigr)^2}_{=:G_2}\\
	+ \underbrace{\bigl((1+\mu_1)(1+\mu_2) - (1+\mu_4)(1+\mu_5)\bigr)^2}_{=:G_3}\biggr]
	\end{multline}
	From the conservation laws \eqref{mass_1} and \eqref{mass_3} 
	we have
	\begin{equation*}\label{t26}
	C_{1,\infty}^2\,\mu_1(\mu_1 + 2) = C_{2,\infty}^2\,\mu_2(\mu_2 + 2).
	\end{equation*}
	Moreover, since $\mu_1,\mu_2\in [-1,+\infty)$, it follows that $\mu_1$ and $\mu_2$ have always a same sign. Similarly, $\mu_4$ and $ \mu_5$ have always a same sign. We have the following useful estimates
	\begin{itemize}
		\item If $\mu_1$ and $\mu_3$ have different signs, then 
		\begin{equation*}\tag{G1}
		G_1 \geq \frac{1}{2}(\mu_1^2 + \mu_2^2 + \mu_3^2).
		\end{equation*}
		Indeed, thanks to $\mu_1\mu_3 \leq 0\leq \mu_1\mu_2$ and  $(1+\mu_1) \geq 0$, we have $(\mu_1 - \mu_3)\mu_2(1+\mu_1) \geq 0$. Hence
		\begin{equation*}
		G_1 = [\mu_1 - \mu_3 + \mu_2(1+\mu_1)]^2 \geq (\mu_1 - \mu_3)^2 \geq \mu_1^2 + \mu_3^2.
		\end{equation*}
		Similarly we get $G_1 \geq \mu_2^2 + \mu_3^2$ and consequently that the estimate (G1) holds.
		\medskip	
		\item If $\mu_4$ and $\mu_3$ have different signs then
		\begin{equation*}\tag{G2}
		G_2 \geq \frac{1}{2}(\mu_4^2 + \mu_5^2 + \mu_3^2).
		\end{equation*}
		This can be proved in the same way as (G1).
		\medskip
		\item If $\mu_1$ and $\mu_4$ have different signs then 
		\begin{equation*}\tag{G3}
		G_3 \geq \frac{1}{3}(\mu_1^2 + \mu_2^2 + \mu_4^2 + \mu_5^2).
		\end{equation*}
		Note that in this case $\mu_1\mu_4 \leq 0\leq \mu_1\mu_2,  \mu_4\mu_5$ and  $(1+\mu_1), (1+\mu_4) \geq 0$ and we have always  
		\begin{equation*}
		(\mu_1 - \mu_4)\mu_2(1+\mu_1) \geq 0, \quad -(\mu_1 - \mu_4)\mu_5(1+\mu_4) \geq 0\; \text{ and } \; -\mu_2\mu_5(1+\mu_1)(1+\mu_4) \geq 0.
		\end{equation*}
		Therefore,
		\begin{equation*}
		G_3 = [\mu_1 - \mu_4 + \mu_2(1+\mu_1) - \mu_5(1+\mu_4)]^2 \geq (\mu_1 - \mu_4)^2 \geq \mu_1^2 + \mu_4^2.
		\end{equation*}
		Similarly we have $G_3 \geq \mu_1^2 + \mu_5^2$ and $G_3\geq \mu_2^2 + \mu_4^2$ and consequently that (G3) holds.
	\end{itemize}		
	\medskip
	From the conservation law \eqref{mass_1} we have the following possibilities concerning the signs of $\mu_1, \mu_3$ and $\mu_4$
	\begin{center}
		\begin{tabular}{ |c|c|c|c| }
			\hline
			Case & $\mu_1$ & $\mu_3$ & $\mu_4$ \\ \hline
			(I) & -  &  - & + \\
			\hline
			(II) & - & +  & -\\
			\hline
			(III) & - & + & + \\ 
			\hline
			(IV) & + & - & - \\
			\hline
			(V) & + & - & +\\
			\hline
			(VI) & + & + & -\\
			\hline
		\end{tabular}
	\end{center}
	We see that in all these cases, there are always two of the three estimates (G1), (G2) and (G3) that hold. Therefore, thanks to \eqref{ttt_1}, we obtain the desired estimate \eqref{ttt}.
\end{proof}
\begin{lemma}\label{lem:lowerbound_chain}
{Under the assumptions of Theorem \ref{mainChain}}, there exists $0<\varepsilon\ll 1$ such that, if $\overline{C_i^2} \leq \varepsilon^2$ for some $i=1,\ldots, 5$, then
	\begin{equation}\label{lowerbound_chain}
		\sum_{i=1}^{5}\|\nabla C_i\|^2 + (\overline{C}_1\overline{C}_2 - \overline{C}_3)^2 + (\overline{C}_4\overline{C}_5 - \overline{C}_3)^2 \geq H_5^{chain}
	\end{equation}
	where
	\begin{equation}\label{H5_chain}
		H_5^{chain} = \min\left\{\frac{C_P\varepsilon^2}{2}; \frac{M_{1,5}}{32}; \frac{M_{1,4}M_{1,5}}{512}; \frac{M_{2,5}^2}{256} \right\}.
	\end{equation}
\end{lemma}
\begin{remark}
	We can compute $\varepsilon$ explicitly, for example
	\begin{equation}\label{epsilon}
	\varepsilon^2 \leq \min\left\{\frac{M_{1,4}}{4}; \frac{M_{1,5}}{4}; \frac{M_{2,5}}{4}; \frac{M_{1,5}}{32M_{2,4}}; \frac{M_{1,4}M_{1,5}}{256M_{2,4}};  \frac{M_{2,5}^2}{256}\right\}.
	\end{equation}
\end{remark}
\begin{proof}[Proof of Lemma \ref{lem:lowerbound_chain}]
	Define the deviations
	\begin{equation*}
		\delta_i(x) = C_i(x) - \overline{C}_i \qquad \text{ for } x\in\Omega, \quad i=1,\ldots, 5.
	\end{equation*}
	We again consider two cases due to various contributions of diffusion and reaction terms.
	\begin{itemize} 
		\item[(i)] (Diffusion is dominant.) If $\|\delta_{i_*}\|^2\geq \varepsilon^2/2$ for some $i_*\in \{1,2,\ldots, 5\}$, then we can estimate
		\begin{equation}\label{lower}
		\text{LHS of (\ref{lowerbound_chain})} \geq C_P\sum_{i=1}^{5}\|\delta_i\|^2 \geq \frac{C_P\,\varepsilon^2}{2}.
		\end{equation}
		\item[(ii)] (Reactions are dominant.) Consider $\|\delta_i\|^2 \leq \varepsilon^2/2$ for all $i=1,2,\ldots, 5$.
		We recall $\overline{C_{i_0}^2} \leq \varepsilon^2$ for some $i_0\in \{1,2,\ldots, 5\}$ and remark that the roles of $C_1, C_2, C_4$ and $C_5$ in \eqref{lowerbound_chain} are the same. Therefore, without loss of generality, it is sufficient to investigate the two cases: $i_0 = 1$ and $i_0 = 3$:
		\begin{itemize}[topsep=5pt, leftmargin=3mm]
			\item[\tiny{$\blacklozenge$}] When $i_0=1$ then from the mass conservation
			\begin{equation*}
			\qquad\overline{C_1^2} + \overline{C_4^2} + \overline{C_3^2} = M_{1,4} \; \qquad\text{ and } \qquad \overline{C_1^2} + \overline{C_5^2} + \overline{C_3^2} = M_{1,5},
			\end{equation*}
			we get
			\begin{equation}\label{k8}
			\qquad\overline{C_3^2} + \overline{C_4^2} \geq M_{1,4} - \varepsilon^2 \geq \frac{M_{1,4}}{2} \qquad \text{ and } \qquad \overline{C_3^2} + \overline{C_5^2} \geq M_{1,5} - \varepsilon^2 \geq \frac{M_{1,5}}{2}.
			\end{equation}
			Without loss of generality, we assume that $M_{1,4} \geq M_{1,5}$. From \eqref{k8} we have the following three possibilities
			\vskip .3cm
			\begin{center}
				\begin{tabular}{ |c|c|c|c| }
					\hline
					Case & $\overline{C_3^2}$ & $\overline{C_4^2}$ & $\overline{C_5^2}$ \\ \hline
					(I) & $\overline{C_3^2}\geq \frac{M_{1,4}}{4}$ & $\leq \frac{M_{1,4}}{4}$ & $\leq \frac{M_{1,5}}{4}$\\
					\hline
					(II) & $\overline{C_3^2}\leq \frac{M_{1,5}}{4}$ & $\geq \frac{M_{1,4}}{4}$ & $\geq \frac{M_{1,5}}{4}$\\
					\hline
					(III) & $\frac{M_{1,5}}{4}\leq \overline{C_3^2} \leq \frac{M_{1,4}}{4}$ & $\geq \frac{M_{1,4}}{4}$ & $\leq \frac{M_{1,5}}{4}$\\
					\hline
				\end{tabular}
			\end{center}
			\vskip .3cm 
			
			In cases (I) and (III), we both have $\overline{C_3^2} \geq \frac{M_{1,5}}{4}$ and, thus
			\begin{equation*}
			\overline{C}_3^2 = \overline{C_3^2} - \|\delta_3\|^2 \geq \frac{M_{1,5}}{4} - \frac{\varepsilon^2}{2} \geq \frac{M_{1,5}}{8}.
			\end{equation*}
			Hence the left hand side of \eqref{lowerbound_chain} is estimated as 
			\begin{equation}
			\text{LHS of (\ref{lowerbound_chain})} \geq(\overline{C}_1\overline{C}_2 - \overline{C}_3)^2
			\geq \frac{1}{2}\overline{C}_3^2 - \overline{C}_1^2\overline{C}_2^2\geq \frac{M_{1,5}}{16} - \varepsilon^2M_{2,4} \geq \frac{M_{1,5}}{32}
			\label{k9}					
			\end{equation}
			thanks to \eqref{epsilon}.
			
			In case (II), we have
			\begin{equation*}
			\overline{C}_4^2 = \overline{C_4^2} - \|\delta_4\|^2 \geq \frac{M_{1,4}}{8} \quad \text{ and } \quad \overline{C}_5^2 = \overline{C_5^2} - \|\delta_5\|^2 \geq \frac{M_{1,5}}{8}.
			\end{equation*}
			We continue with
			\begin{equation}\label{k10}
			\begin{aligned}
			\text{LHS of (\ref{lowerbound_chain})}&\geq  (\overline{C}_1\overline{C}_2 - \overline{C}_3)^2 + (\overline{C}_4\overline{C}_5 - \overline{C}_3)^2
			\geq \frac{1}{2}(\overline{C}_4\overline{C}_5 - \overline{C}_1\overline{C}_2)^2\\
			&\geq \frac{1}{4}\overline{C}_4^2\overline{C}_5^2 - \frac{1}{2}\overline{C}_1^2\overline{C}_2^2 
			\geq \frac{M_{1,4} M_{1,5}}{256} - \frac{1}{2}\varepsilon^2M_{2,4} \geq \frac{M_{1,4}M_{1,5}}{512}
			\end{aligned}
			\end{equation}
			thanks again to \eqref{epsilon}. Combining \eqref{k9} and \eqref{k10}, we have 
			\begin{equation}\label{k11}
			\begin{gathered}
			\text{LHS of (\ref{lowerbound_chain})}
			\geq \min\left\{\frac{M_{1,5}}{32}; \frac{M_{1,4}M_{1,5}}{512}\right\}
			\end{gathered}
			\end{equation}
			which ends the proof in the case $i_0=1$.
			\item[\tiny{$\blacklozenge$}] When $i_0 = 3$, we obtain first that $\overline{C}_{3}^2 \leq \overline{C_{3}^2} \leq \varepsilon^2$.
			
			Without loss of generality, we can assume that $M_{1,4}$ is the biggest component of $\M$. Thus,
			\begin{equation*}
			\overline{C_1^2} = \overline{C_2^2} + M_{1,4} - M_{2,4} \geq \overline{C_2^2} \quad \text{ and } \quad \overline{C_4^2} = \overline{C_5^2} + M_{1,4} - M_{1,5} \geq \overline{C_5^2}.
			\end{equation*}
			By using the mass conservation $\overline{C_2^2} + \overline{C_3^2} + \overline{C_5^2} = M_{2,5}$, we get
			\begin{equation*}
			\overline{C_2^2} + \overline{C_5^2} \geq M_{2,5} - \varepsilon^2 \geq \frac{M_{2,5}}{2},
			\end{equation*}
			hence
			\begin{equation*}
			\overline{C_2^2} \geq \frac{M_{2,5}}{4} \quad \text{ or } \quad \overline{C_5^2} \geq \frac{M_{2,5}}{4}.
			\end{equation*}
			If $\overline{C_2^2} \geq \frac{M_{2,5}}{4}$ then $\overline{C_1^2} \geq \frac{M_{2,5}}{4}$. It follows that
			\begin{equation*}
			\overline{C}_1^2 = \overline{C_1^2} - \|\delta_1\|^2 \geq \frac{M_{2,5}}{8} \quad \text{ and } \quad
			\overline{C}_2^2 = \overline{C_2^2} - \|\delta_2\|^2 \geq \frac{M_{2,5}}{8}.
			\end{equation*}
			We can then estimate due to \eqref{epsilon}
			\begin{equation}\label{k12}
			\text{LHS of (\ref{lowerbound_chain})} \geq (\overline{C}_1\overline{C}_2 - \overline{C}_3)^2 \geq \frac{1}{2}\overline{C}_1^2\overline{C}_2^2 - \overline{C}_3^2
			\geq \frac{M_{2,5}^2}{128} - \varepsilon^2 \geq \frac{M_{2,5}^2}{256}.
			\end{equation}
			Similarly, if $\overline{C_5^2} \geq \frac{M_{2,5}}{8}$, we can prove by using the same arguments above that
			\begin{equation}\label{k13}
			\text{LHS of (\ref{lowerbound_chain})} \geq \frac{M_{2,5}^2}{256},
			\end{equation}
			which ends the case $i_0=3$.
		\end{itemize}
		From \eqref{lower}, \eqref{k11}, \eqref{k12} and \eqref{k13}, we obtain the desired estimate \eqref{lowerbound_chain}.
	\end{itemize}
\end{proof}
\vskip .2cm
\begin{proof}[Proof of Theorem \ref{mainChain}]
	First of all, the validity of assumptions $(\mathbf{A1})$ and $(\mathbf{A2})$ for system \eqref{SysChain} follows from Lemma \ref{Chain_Equi}.
	
	Then, the entropy entropy-dissipation estimate \eqref{l5} follows for a constant $\lambda_2$ from Theorem \ref{the:main}, Lemma \ref{Chain_InterInequal} and Lemma \ref{lem:lowerbound_chain}.	
From \eqref{SysChain}, we have
	\begin{equation*}
		\partial_t(c_1 + c_3 + c_4) - \Delta(d_1c_1 + d_3c_3 + d_4c_4) = 0
	\end{equation*}
	and
	\begin{equation*}
		\partial_t(c_2 + c_3 + c_5) - \Delta(d_2c_2 + d_3c_3 + d_5c_5) = 0.
	\end{equation*}
	Then, it follows from an improved duality estimate (see \cite[Lemma 3.2]{CDF14}, which holds in fact in all space dimensions) that for some $p>2$ (sufficiently close to 2)
	\begin{equation*}
		\|c_i\|_{L^p(0,T;L^p(\Omega))} \leq C(T) \quad\text{ for all } T>0
	\end{equation*}
where $C(T)$ is a constant, which grows at most polynomially in $T$. Since all the nonlinearities in \eqref{SysChain} are quadratic, we can apply results from e.g. \cite{Michel,DFPV} to ensure that there exists a global weak solution $\cc = (c_1, \ldots, c_5) \in L^p(0,T;L^p(\Omega))^5$ for all $T>0$.
Moreover, similarly to \cite{DFPV} each weak solution satisfies the weak entropy entropy-dissipation law
	\begin{equation*}
		\mathcal{E}(\cc(t)) + \int_s^t\mathcal{D}(\cc(\tau))d\tau = \mathcal{E}(\cc(s)) \quad  \text{ for a.a. } t>s>0.
	\end{equation*}	
	Hence, by applying the entropy entropy-dissipation inequality \eqref{l5} to the above entropy entropy-dissipation law, a Gronwall argument yields exponential convergence to equilibrium in relative entropy
	\begin{equation*}
		\mathcal{E}(\cc(t)) - \mathcal{E}(\ww) \leq e^{-\lambda_2t}(\mathcal{E}(\cc_0) - \mathcal{E}(\ww))
	\end{equation*}
	and finally, by using the Csisz\'{a}r-Kullback-Pinsker inequality in Lemma \ref{CKP}, we obtain exponential $L^1$-convergence of weak solutions to equilibrium
	\begin{equation*}
		\sum_{i=1}^{5}\|c_i(t) - c_{i,\infty}\|^2_{L^1(\Omega)} \leq C^{-1}_{CKP}(\mathcal{E}(\cc_0) - \mathcal{E}(\ww))e^{-\lambda_2t}\quad \text{ for all } t>0.
\end{equation*}
%
\end{proof}

\section{Summary, further applications and open problems}\label{sec:5}
In this paper, we exploit the entropy method to show explicit convergence to equilibrium for detailed balance chemical reaction-diffusion networks describing substances in a bounded domain $\Omega\subset \mathbb{R}^n$ according to the mass action law. More precisely, in Section \ref{subsec:2.2} we propose a constructive method to prove an entropy entropy-dissipation estimate with computable rates and constants. The corresponding If-Theorem \eqref{the:main} relies on two assumed functional inequalities \eqref{ODEEqui} and \eqref{FarEqui}, where the first one is a finite dimensional inequality and quantifies the uniqueness of the detailed balance equilibrium. The second inequality is a lower bound for the entropy dissipation when the concentration vector is close to the boundary $\partial\mathbb{R}^{I}_{+}$. These estimates suggest a general method for any general system of the form \eqref{SS}--\eqref{R}, once the mass conservation laws are explicitly known.
The applicability of the general method is demonstrated in Sections \ref{sec:3} and \ref{sec:4} 
for two specific types of reaction-diffusion networks:  
A single reversible reaction \eqref{single} and a chain of reversible reactions \eqref{ChainReaction}, which is motivated by enzyme reactions.

\subsection{Further applications}\label{further}
We point out  that the proposed approach applies also to reaction-diffusion networks where the chemical substances are supported on different spatial compartments. For example in \cite{BFL}, the reversible reaction 
\begin{equation*}
\alpha\, \mathcal{U} \leftrightharpoons \beta\, \mathcal{V}
\end{equation*}
is considered between a bounded domain $\Omega\subset\mathbb{R}^n$ and its smooth boundary $\Gamma:= \partial\Omega$, where $\mathcal{U}$ is the volume-substance inside $\Omega$ and $\mathcal{V}$ is the surface-substance on $\Gamma$ and the reaction is assumed to happen on $\Gamma$. The corresponding volume-surface reaction-diffusion system reads as
\begin{equation}\label{vol-surf}
\begin{cases}
u_t - d_u\Delta u = 0, &x\in\Omega, \quad t>0,\\
d_u\partial_{\nu}u = -\alpha(u^\alpha - v^\beta), &x\in\Gamma, \quad t>0,\\
v_t - d_v\Delta_{\Gamma}v = \beta(u^\alpha - v^\beta), &x\in\Gamma, \quad t>0,\\
u(0,x) = u_0(x), \quad v(0,x) = v_0(x),&x\in\Omega,
\end{cases}
\end{equation}
in which $u: \Omega\times \mathbb{R}_+ \rightarrow \mathbb{R}_+$ is the volume-concentration of $\mathcal{U}$ and $v: \Gamma\times \mathbb{R}_+ \rightarrow \mathbb{R}_+$ is the surface-concentration of $\mathcal{V}$, and $\Delta_{\Gamma}$ is the Laplace-Beltrami operator which describes diffusion of $\mathcal{V}$ along $\Gamma$. The system \eqref{vol-surf} possesses the mass conservation law
\begin{equation*}
									\beta\int_{\Omega}u(x,t)dx + \alpha\int_{\Gamma}v(x,t)dS = \beta\int_{\Omega}u_0(x)dx + \alpha\int_{\Gamma}v_0(x)dS =: M >0, \qquad \forall\ t\ge0
\end{equation*}
and thus obeys a unique positive equilibrium $(u_{\infty}, v_{\infty})$ satisfying
						\begin{equation*}
									\begin{cases}
												u_{\infty}^{\alpha} = v_{\infty}^{\beta},\\
												\beta|\Omega|u_{\infty} + \alpha|\Gamma|v_{\infty} = M.
									\end{cases}
						\end{equation*}
						To show the convergence to equilibrium for \eqref{vol-surf}, we consider the entropy functional
						\begin{equation*}
									\mathcal{E}(u,v) = \int_{\Omega}(u\log u - u + 1)dx + \int_{\Gamma}(v\log v - v + 1)dS
						\end{equation*}
						and its entropy-dissipation
						\begin{equation*}
									\mathcal{D}(u,v) = d_u\int_{\Omega}\frac{|\nabla u|^2}{u}dx + d_u\int_{\Gamma}\frac{|\nabla_{\Gamma}v|^2}{v}dS + \int_{\Gamma}(u^{\alpha} - v^{\beta})\log{\frac{u^{\alpha}}{v^{\beta}}}dS.
						\end{equation*}
						The aim is to prove an EED estimate of the form
						\begin{equation}\label{vs-EED}
									\mathcal{D}(u,v) \geq \lambda(\mathcal{E}(u,v) - \mathcal{E}(u_{\infty}, v_{\infty})),
						\end{equation}
						for all $(u,v)$ satisfying the mass conservation $\beta\int_{\Omega}u(x)dx + \alpha\int_{\Gamma}v(x)dS = M$.
						\vskip .2cm
						\noindent The EED estimate \eqref{vs-EED} can be proved by applying the method proposed in Section \ref{sec:2} with few changes, e.g. the Poincar\'{e} inequality is replaced by the Trace inequality $\|\nabla f\|_{L^2(\Omega)}^2 \geq C_T\|f - \overline{f}\|_{L^2(\Gamma)}^2$. The interested reader is referred to \cite{BFL} for more details.

In fact, it's easy to verify that the results of \cite{BFL} are also applicable in the following network
\begin{equation*}
\alpha_1\mathcal U_1 + \ldots + \alpha_I\mathcal U_I \leftrightharpoons \beta_1\mathcal V_1 + \ldots + \beta_J\mathcal V_J
\end{equation*}
in the case that $\mathcal{U}_{i=1,\ldots, I}$ are volume-concentrations on $\Omega$ and $\mathcal{V}_{j=1,\ldots, J}$ are surface-concentrations on $\partial\Omega$.
						
			\subsection{Open problems}
There are many open problems related to the results of this paper. Here, we list the two questions we find the most interesting:
\begin{itemize}
\item [1.] How to choose the conservation laws in the general case? \\
As mentioned in the introduction, the conservation laws $\mathbb{Q}\, \overline{\cc} = \mathbf{M}$ depend on the choice of the matrix $\mathbb{Q}$, which has rows forming a basis of $\mathrm{ker}(W)$, where $W$ is the Wegscheider matrix. The choice of $\mathbb{Q}$ is not unique and in fact, there are infinitely many matrices like $\mathbb{Q}$. The question is: can we have a procedure or a method to choose such a matrix $\mathbb{Q}$ for general systems, which is suitable for our method and allows to explicitly complete the proof of \eqref{ODEEqui} and \eqref{FarEqui} in the general case?
									
									\item[2.] How to get optimal convergence rate?\\
We made it clear in this paper (see Remark \ref{rem:explicit}) that although we obtain an explicit bound for the convergence rate, the convergence rate in this work is non-optimal. The question of optimal convergence rate using the entropy method is highly involved and left for future investigation. 
\end{itemize}
			
\vskip 0.5cm
\noindent{\bf Acknowledgements.} The second author is supported by International Research Training Group IGDK 1754. This work has partially been supported by NAWI Graz. K.F.
acknowledges the ESI Thematic Programme on
Nonlinear Flows. The work in this paper has certainly benefited from the discussions at the ESI Vienna. 
Finally, the authors would like to thank the referees for helping to improve the article.

\end{document}